\documentclass[english,reqno]{amsart}

\usepackage{amsmath}
\usepackage{amssymb} 
\usepackage{enumitem} 
\usepackage{mathtools} 
\usepackage[table]{xcolor} 
\usepackage[all]{xy} 
\usepackage{tikz} 
\usepackage{tikz-cd}
\usepackage{indentfirst} 
\usepackage{babel} 
\usepackage{setspace} 

\usepackage[colorlinks,linkcolor=red,anchorcolor=green,citecolor=blue]{hyperref} 
\hypersetup{linktocpage = true} 

\usepackage{rotating} 

\usepackage{ytableau} 
\usepackage{longtable} 
\newcolumntype{M}[1]{>{\centering\arraybackslash}m{#1}} 

\usepackage{mathpazo}
\usepackage{mathrsfs}
\DeclareFontFamily{OMS}{rsfs}{\skewchar\font'60}
\DeclareFontShape{OMS}{rsfs}{m}{n}{<-5>rsfs5 <5-7>rsfs7 <7->rsfs10 }{}
\DeclareSymbolFont{rsfs}{OMS}{rsfs}{m}{n}
\DeclareSymbolFontAlphabet{\scr}{rsfs}
\DeclareSymbolFontAlphabet{\scr}{rsfs}

\usepackage[T1]{fontenc}

\newcommand\cA{{\mathcal A}}

\newcommand\cE{{\mathcal E}}
\newcommand\cF{{\mathcal F}}
\newcommand\cG{{\mathcal G}}
\newcommand\cH{{\mathcal H}}

\newcommand\cK{{\mathcal K}}
\newcommand\cL{{\mathcal L}}
\newcommand\cN{{\mathcal N}}

\newcommand\cO{{\mathcal O}}

\newcommand\cQ{{\mathcal Q}}

\newcommand\cU{{\mathcal U}}
\newcommand\cV{{\mathcal V}}
\newcommand\cW{{\mathcal W}}

\DeclareMathOperator*{\Sym}{Sym}

\theoremstyle{plain}
\newtheorem{thm}{Theorem}[section]
\newtheorem{lemma}[thm]{Lemma}
\newtheorem{prop}[thm]{Proposition}
\newtheorem{cor}[thm]{Corollary}

\theoremstyle{remark}
\newtheorem{example}[thm]{Example}
\newtheorem{remark}[thm]{Remark}

\setlist[itemize]{leftmargin=*}
\setlist[enumerate]{leftmargin=*}

\numberwithin{equation}{section} 

\makeatletter

\ifnum\@ptsize=0  \fi \ifnum\@ptsize=2  \fi

\title{Foliations whose first Chern class is nef } 


\author{Wenhao Ou}

\address{Wenhao Ou, Institute of Mathematics, Academy of Mathematics and Systems Science, Chinese Academy of Sciences, Beijing, 100190, China}
\email{wenhaoou@amss.ac.cn}

\begin{document}

\begin{abstract}
  Let $\cF$ be a foliation on a projective manifold $X$ with  $-K_\cF$ nef. 
  Assume that either $\cF$ is regular, or it has a compact leaf. 
  We prove that  there is a locally trivial fibration $f\colon X\to Y$, and  a foliation $\cG$ on $Y$ with $K_{\cG} \equiv 0$, such that $\cF = f^{-1}\cG$. 
 \end{abstract}

\maketitle
 
\tableofcontents


\section{Introduction}

From the viewpoint of minimal model program, complex projective manifolds $X$ could be classified according to the numerical behavior of their canonical classes $K_X$. When $K_X\equiv 0$, Yau proved the Calabi conjecture, which confirms the existence of K\"ahler–Einstein metrics on $X$. 
The universal cover of $X$ then admits a decomposition, known as Beauville-Bogomolov decomposition (\textit{e.g.} \cite{Beauville1983}), as the product of an affine space, of holomorphic symplectic manifolds, and of simple Calabi-Yau manifolds.  

In a more general setting, if we assume that $-K_X$ is nef, then Demailly-Peternell-Schneider  conjectured an analogue uniformization. 
The conjecture was recently settled in \cite{Cao19}  and \cite{CaoHoering2019}, based on the positivity of direct images  established in \cite{PaunTakayama2018}. 
More precisely, the universal cover of $X$ is isomorphic to the  product of an affine space, of holomorphic symplectic manifolds,  of simple Calabi-Yau manifolds, and of a rationally connected projective manifold with nef anticanonical class.

In \cite{Cao19} mentioned above,  the Albanese morphism   $\mathrm{alb}_X \colon X \to \mathrm{Alb}(X)$ was investigated. 
When $-K_X$ is nef, such a morphism is known to be a fibration after \cite{Zhang1996}. 
It was also proved   that $\mathrm{alb}_X$ is semistable in codimension one  (see \cite{Zhang2005})  and   equidimensional (see \cite{LTZZ2010}).  
In particular, the relative canonical class $K_{X/\mathrm{Alb}(X)}$ is the same as $K_{\cF}$, the canonical class  of the foliation $\cF$ induced by $\mathrm{alb}_X$. 
Since $\mathrm{Alb}(X)$ is an abelian variety, the nefness of $-K_X$ is then transposed to the nefness of   $-K_{\cF}$.

Under such   consideration, it is natural to expect similar structural results in the context of foliations on projective manifolds. 
Indeed, foliations with zero canonical class have already been broadly studied (\textit{e.g.} \cite{Touzet2008},  \cite{PereiraTouzet2013}, \cite{LorayPereiraTouzet2018}, \cite{Druel2021},  \cite{DruelOu2019} and \cite{DPPT21}).  
Foliations with nef anticanonical class have been previously investigated as well (\textit{e.g.} \cite{Druel2017}, \cite{Druel2019} ,  \cite{CaoHoering2019} and \cite{CCM19}). Particularly, for regular foliations with semipositive anticanonical class, Druel proved in \cite{Druel2019} that one can reduce the problem to the case of   foliations  with zero canonical class. 
An analogue statement for foliations with nef anticanonical class was conjectured by Cao and H\"oring. 
The main objective of this paper is   to prove this conjecture.

\begin{thm}\label{thm:main-theorem}
Let $\cF$ be a foliation on a projective manifold  $X$ with  $-K_\cF$ nef. Assume that either $\cF$ is regular, or it has a compact leaf.
Then  there is a locally trivial fibration $f\colon X\to Y$ with rationally connected   fibers. 
Moreover, there is  a foliation $\cG$ on $Y$ with $K_{\cG} \equiv 0$ such that $\cF = f^{-1}\cG$.
\end{thm}

   
We also obtain the following corollary.

\begin{cor}\label{cor:regularity}
Let $\cF$ be a foliation on a projective manifold $X$ with $-K_{\cF}$ nef. Assume that $\cF$ has a compact leaf. Then it is regular and there is a regular foliation $\cE$ such that $T_X=\cF\oplus \cE$. 
\end{cor}

A key  ingredient for Theorem \ref{thm:main-theorem} is the following theorem, which reduces the situation to the case of  algebraically integrable foliations.  
It follows directly from a series of  Druel's works.

\begin{thm}\label{thm:rc-reduction}
Let $\cF$ be a foliation on a projective manifold $X$ with nef anticanonical class $-K_\cF$. Assume that either $\cF$ is regular or it has a compact leaf.  Then the following properties hold.
\begin{enumerate}
\item The algebraic part $\cF_{alg} \subseteq \cF$ has a compact leaf, and is induced by a dominant  almost holomorphic rational map $\varphi \colon X\dashrightarrow Y$.
\item Let $\cF_{rc}$ be the foliation of the relative MRC fibration of $\varphi$. Then $K_{\cF_{rc}} \equiv K_{\cF_{alg}} \equiv K_{\cF}$.
\end{enumerate}
\end{thm}

For algebraically integrable foliations,   we will prove the  following statements, which extend the results in \cite{CaoHoering2019} and \cite{CCM19}.  
Together with Theorem \ref{thm:rc-reduction}, they will imply Theorem \ref{thm:main-theorem}.

\begin{thm}\label{thm:decomposition}
Let $\cF$ be an algebraically integrable foliation on a projective manifold $X$ with $-K_{\cF}$ nef. Assume that $\cF$ has a compact leaf. Then there is a foliation $\cG$ on $X$ such that $T_X=\cF \oplus \cG$. 
\end{thm}

\begin{cor}\label{cor:structure}
Let $\cF$ be an algebraically integrable foliation on a projective manifold $X$ with $-K_{\cF}$ nef. Assume that $\cF$ has a compact leaf. Then it is induced by an equidimensional fibration $f\colon X \to Y$. If moreover $\cF$ is rationally connected, then $f$ is   a locally trivial family. 
\end{cor}

The paper is organized as follows. We will work over $\mathbb{C}$, the field of complex numbers, throughout. 
In Section \ref{section:foliation}, we  will recall the basic language on foliations. 
In  Section \ref{section:flat}, we  will  prove several elementary results about numerically flat vector bundles. 
In Section \ref{section:semistable}, we  will collect some powerful tools developed in  \cite{CCM19}.  
Then we will construct an appropriate semistable reduction for a foliation with nef aniticanonical class in Section \ref{section:construction}.
In the end, we will finish the proofs of the theorems in Section \ref{section:proof}.\\

\noindent \textbf{Acknowledgment} The author is very grateful to Jie Liu for inspiring conversations. He would like to thank Junyan Cao and Andreas H\"oring  for remarks and comments. He would like to express his   gratitude to St\'ephane Druel for his enormous help in improving the presentation of the paper. 
He would like to thank the referee for reading the paper carefully and providing helpful suggestions. 
The author is supported by the National Key R\&D Program of China (No. 2021YFA1002300).

\section{Foliations}\label{section:foliation}

In this section, we gather some basic properties concerning foliations on complex algebraic varieties.  
For more details on the notion of foliations, we refer to  \cite[Section 3]{AraujoDruel2014}.

A  {foliation} on  a normal variety $X$ is a coherent subsheaf $\cF\subseteq T_X$, where $T_X =(\Omega_X^1)^*$ is the reflexive tangent sheaf, such that
\begin{enumerate}
\item $\cF$ is closed under the Lie bracket,  
\item $\cF$ is saturated in $T_X$, that is,  the quotient $T_X/\cF$ is torsion-free.
\end{enumerate} 
The  {canonical class} $K_{\cF}$ of $\cF$ is any Weil divisor on $X$ such that  $\cO_X(-K_{\cF})\cong \det\cF$. 
In particular, the first Chern class $c_1(\cF)$ is equivalent to $-K_\cF$.
Let $X^\circ$ be the largest open subset of the smooth locus of $X$ such that  $\cF|_{X^\circ}$ is a subbundle of $T_{X^\circ}$.    
The singular locus of $\cF$ is defined to be $X \setminus X^\circ$. 
When $X=X^\circ$, we say that $\cF$ is a regular foliation of $X$. 
A leaf of $\cF$ is a maximal connected and immersed holomorphic submanifold $i\colon L \hookrightarrow X^\circ$  such that the differential map 
\[
\mathrm{d}i\colon T_L \to i^*\cF
\]
is an isomorphism. 
A leaf is called algebraic if it is open in its Zariski closure in $X$. 
The foliation $\cF$ is said to be  {algebraically integrable} if its leaves are algebraic.

Let $\cF$ be a codimension $q$ foliation on an $n$-dimensional normal variety
$X$. 
The normal sheaf $\cN_{\cF}$ of $\cF$ is the reflexive hull $(T_X/\cF)^{**}$.  
The $q$-th wedge product of the inclusion $\cN_{\cF}^* \to (\Omega_X^{1})^{**}$
 gives
rise to a non-zero global section $\omega \in H^0(X, (\Omega^{q}_X \otimes
\mathrm{det}\, \cN_{\cF} )^{**}),$ whose zero locus has codimension at least two in
$X$.  
 Moreover, $\omega$ is locally decomposable and integrable. 
To say that $\omega$ is locally decomposable means that, in a
neighborhood of a general point of $X$, $\omega$ decomposes as the wedge product of $q$ local $1$-forms 
$\omega = \omega_1 \wedge \cdots \wedge \omega_q$. 
To say that it is integrable means that for this local decomposition one has $\mathrm{d} \omega_i \wedge \omega = 0$  for every $i \in \{1,..., q\}$. 

The integrability condition for $\omega$ in the previous paragraph       is equivalent to the condition that $\cF$ is closed under the Lie bracket.
Conversely, let $\cL$ be a reflexive sheaf of rank  one  on $X$, and let $\omega \in H^0(X, (\Omega_X^q\otimes \cL)^{**}) $  be a global section whose
zero locus has codimension at least two in $X$. 
Suppose that $\omega$ is locally decomposable and integrable. 
Then the
kernel of the morphism $T_X  \to  (\Omega_X^{q-1} \otimes \cL)^{**}$
  given by the contraction with $\omega$ defines a foliation of codimension $q$
on $X$. 
These constructions are inverse of each other.

Let $X$ and $Y$ be normal varieties, and  let $\varphi \colon X\dashrightarrow Y$ be a dominant rational map that restricts to a morphism $f^\circ\colon X^\circ\to Y^\circ$,
where $X^\circ\subseteq X$ and  $Y^\circ\subseteq Y$ are smooth open dense subsets.
Let $\cG$ be a   foliation on $Y$.  Then the pullback foliation $\varphi^{-1}\cG$ on $X$ is induced by $(\mathrm{d}f)^{-1}(f^*\cG)$ on $X^\circ$, where 
\[\mathrm{d}f \colon T_{X^\circ} \to f^*T_{Y^\circ}\] is the differential map of $f$.

Assume that $\cF$ is an algebraically integrable foliation on a normal variety $X$. 
Then  there is a unique normal variety $V$ in the normalization of the Chow variety of $X$, whose general points parametrize  the closures of general leaves of $\cF$. 
Let $U$ be the normalization of the universal family over $V$. 
Then the natural projective fibration  $U\to V$ is called the family of leaves of $\cF$. For more details, see for example \cite[Lemma 3.2]{AraujoDruel2013}.

As a consequence, a foliation $\cF$  is  algebraically integrable if and only if there is a  dominant rational map $\varphi\colon X\dashrightarrow Y$ to a normal variety such that $\cF = \varphi^{-1}0_Y$, where $0_Y$ is the foliation by points on $Y$. 
In this case, we say that $\cF$ is the foliation induced by $\varphi$.

For foliations induced by equidimensional fibrations, we have the following  fact.

\begin{lemma}\label{lemma:basechange-canonical}
Let $\varphi \colon X\to Y$ be an equidimensional morphism between normal varieties, and $\cF$ the foliation induced. 
Let $Y' \to Y$ be a finite dominant morphism with $Y'$ normal. 
Let $X'$ be the normalization of $X\times_Y Y'$. If $\cF'$ is the foliation induced by the natural projection $X'\to Y'$, and if $p\colon X' \to X$ is the natural finite morphism, then $K_{\cF'} = p^*K_\cF$.
\end{lemma}

\begin{proof}
Let $D$ be any prime divisor in $X$. 
We only need to show that  $K_{\cF'} = p^*K_{\cF}$ around a general point of $p^{-1}(D)$. 

If $D$ is not contained in the branched locus of $p$, then $p$ is locally a diffeomorphism 
around a general point of $p^{-1}(D)$.  
Thus  $K_{\cF'} = p^*K_{\cF}$ around a general point of $p^{-1}(D)$. 
 
Assume that $D$ is contained in the branched locus of $p$. 
In this case, $D$ must be $\cF$-invariant by the construction of $X'$. 
We can then  apply  \cite[Lemma 3.4]{Druel2021} to show that $K_{\cF'} = p^*K_{\cF}$ around a general point of $p^{-1}(D)$.  
This completes the proof of the lemma.
\end{proof}

Let $\cF$ be a  foliation   on a normal variety $X$. There exist a normal variety $Y$, unique up to birational equivalence, a dominant rational map with connected fibers $\varphi\colon X\dashrightarrow Y$, and a foliation $\cG$ on $Y$,  such that the following properties hold.
\begin{enumerate}
\item $\cG$ is purely transcendental, that is, there is no positive-dimensional algebraic subvariety through a general point of $Y$ which is tangent to $\cG$.
\item $\cF = \varphi^{-1}\cG$.
\end{enumerate} 
The foliation induced by   $\varphi$ is called the algebraic part of $\cF$. 
For more details, see \cite[Section 2.3]{LorayPereiraTouzet2018}.

We will need the following lemma for the proofs of the main theorems.

\begin{lemma}\label{lemma:descend-splitting}
Let $ {f} \colon  {X} \to  {Y}$ be a smooth dominant morphism between complex manifolds. 
Let $\cH$ be a foliation on $ {X}$  whose rank is equal to $\dim  {Y}$. 
Let $q\colon  {Y}'\to  {Y}$ be a finite surjective morphism with $Y'$ smooth, and let $ {f}'\colon  {X}' \to  {Y}'$ be the base change of $ {f}$ over $ {Y}'$. 
Let $\cH'= {f}^{-1}\cH$ be the pullback foliation. 
If $\cH'$ is regular and transverse to every fiber of $ {f}'$, then $\cH$ is regular and  transverse to every fiber of $ {f}$.
\begin{equation*} 
\begin{tikzcd}[column sep=large, row sep=large]
   {X}' \ar[d," {f}'"]     \ar[r," {p}"]  &  {X} \ar[d," {f}"]  \\
   {Y}' \ar[r," {q}"]&  {Y} 
\end{tikzcd}
\end{equation*}
\end{lemma}

\begin{proof} 
Set  $n=\dim X$ and $d=\dim  {Y}$.   
We  first notice that it suffices to prove the assertion in codimension one of $X$. 
Indeed, let $U\subseteq X$ be an  open subset  whose complement  has  codimension  at least two in $X$. 
If we can prove that $\cH|_U$ is transverse to $T_{X/Y}|_U$, that is $T_U = T_{X/Y}|_U \oplus \cH|_U$, 
then we must have $T_X = T_{X/Y} \oplus \cH$, for $\cH$, $T_{X/Y}$ and $T_X$ are all reflexive sheaves (see \cite[Proposition 1.6]{Hartshorne1980}).   
Since $f$ is a smooth morphism, $T_{X/Y}$ is locally free. 
Thus  this decomposition of $T_X$  will imply that $\cH$ is  regular and transverse to all fibers of $f$.

Hence by shrinking $X$,  we may  assume that $\cH$ is a subbundle of $T_X$. 
Let $D\subseteq X$ be the branched locus of $p$. 
Then by assumption, we see that none of the components of $D$ is   $\cH$-invariant. 
Therefore, thanks to \cite[Lemma 3.4]{Druel2021},  we get $p^*(\mathrm{det}\,  \cN_{\cH} )\cong \mathrm{det}\, \cN_{\cH'}$, 
where $\cN_{\cH}$ and $\cN_{\cH'}$ are the normal bundles of $\cH$ and $\cH'$ respectively.  

Let $x\in X$ be a point. We will show that $T_{X/Y}$ and $\cH$ are transverse at $x$. 
The problem is now local around $x$. 
By shrinking $X$ around $x$, we may assume that   there is a nowhere vanishing  holomorphic $(n-d)$-form $\omega$ such that $\cH$ is the kernel of the contraction by $\omega$. 
Then the conclusion of the previous paragraph implies that $\omega'=p^*\omega$ is nowhere vanishing.

It remains to show that, for every local smooth vector field $v$ on $X$, nowhere vanishing and with values in $T_{X/Y}$, the contraction $v\lrcorner \omega$ is nowhere vanishing. 
Since $f'$ is the  base change of $f$, it follows that there is a   smooth vector field $v'$ on $X'$, nowhere vanishing and with values in $T_{X'/Y'}$,  such that $\mathrm{d}p(v') = v$.  
Moreover, by assumption, $v'\lrcorner \omega'$ is nowhere vanishing.  
Since $v'\lrcorner \omega' = p^*(v\lrcorner \omega)$, it follows that  
$v\lrcorner \omega$ is nowhere vanishing. 
This completes the proof.
\end{proof}

\section{Numerically flat vector bundles on complex manifolds}\label{section:flat}

The notion of numerically flat vector bundles was introduced in \cite{DemaillyPeternellSchneider1994}. 
It plays an important role in the proof of locally trivial fibrations. 
A vector bundle $\cE$ on a projective manifold is numerically flat  if  both $\cE$ and $\cE^*$ are nef. 
We recall the following criterion of numerically flat vector bundles for reflexive sheaves (see  \cite[Corollary 3]{BandoSiu1994} and  \cite[Theorem IV.4.1]{Nakayama2004})

\begin{thm}
\label{thm:num-flat-criterion-1}
Let $X$ be a projective manifold of dimension $n$ and $\cE$ a reflexive sheaf on $X$.  Let $H$ be an ample divisor in $X$. Then $\cE$ is a numerically flat vector bundle if and only if the following conditions hold.
\begin{enumerate}
\item $\cE$ is $H$-semistable.
\item $c_1(\cE)\cdot H^{n-1} = c_1(\cE)^2\cdot H^{n-2}= c_2(\cE)\cdot H^{n-2} = 0$. 
\end{enumerate} 
\end{thm}

A torsion free (coherent) sheaf $\cE$ on a complex K\"ahler manifold $(Y, \omega)$ is called weakly positively curved if the following statement holds. 
Let $Y^\circ$ be the locally free locus of $\cE$. For every $\epsilon >0$, there is a possibly singular Hermitian metric $h_\epsilon$ of $\cE|_{Y^\circ}$ such that  
\[ \sqrt{-1}\Theta_{h_\epsilon}(\cE) \succeq -\epsilon \omega \otimes \mathrm{Id}_{\cE} \mbox{ on } Y^\circ.\]
For more details, we refer to \cite[Section 2]{CaoHoering2019} and \cite[Section 2]{PaunTakayama2018}.
The following lemma was essentially proved in \cite[Theorem 2.5.2]{PaunTakayama2018}, \cite[Theorem 2.21]{Paun2016} and \cite[Lemma 2.10]{CaoHoering2019}.

\begin{lemma}\label{lemma:weakly-positive}
Let $\cE$ be a torsion free weakly positively curved sheaf on a projective manifold $X$. 
Let $A$ be an ample divisor and $a>0$ an integer.
Then there exists an integer $m>0$, such that the natural evaluation map
\[H^0(X,{(\Sym}^{am} \cE)^{**} \otimes \cO_X(mA)) \to  ({\Sym}^{am} \cE)^{**}_x \otimes \cO_X(mA)_x \]
is surjective for general $x\in X$.
\end{lemma}

\begin{proof}
We follow the lines   of \cite[Theorem 2.21]{Paun2016}. 
We denote by $X_\cE$  the locally free locus of $\cE$. 
Let  $\mu\colon Y \to \mathbb{P}(\cE)$ be a desingularization of the main component of  $\mathbb{P}(\cE)$,  
which is an isomorphism over $X_\cE$. 
We set  $P$ the pullback of the tautological line bundle $\cO_{\mathbb{P}(\cE)}(1)$ to $Y$.
Let $f\colon Y \to X$ be the natural morphism and  $Y_1=f^{-1}(X_\cE)$.
Then $Y_1\cong \mathbb{P}(\cE|_{X_\cE})$. 
We consider the line bundle $L=P^{\otimes 2a}\otimes f^*\cO_X(A)$. 

Since $\cE$ is weakly positively curved, $L|_{Y_1}$ admits a possibly singular metric $h_L$ which has semipositive curvature.
Moreover, $h_L$ is bounded along a general fiber $Y_x$ of $f$.
In particular,   the multiplier ideal sheaf of $h_L^{\otimes k}|_{Y_{x}}$ is trivial for any integer $k$ large enough.
Arguing as in \cite[Theorem 2.21]{Paun2016}, we obtain  that  for some integer $k$ large enough, the reflexive sheaf
$$({(\Sym}^{2ak} \cE)^{**} \otimes \cO_X(kA)) \otimes \cO_X(kA) = {(\Sym}^{2ak} \cE)^{**} \otimes \cO_X(2kA) $$
is generated by global sections over some open dense subset of $X$.
We can then deduce the lemma by letting $m=2k$. 
\end{proof}

The following assertion is an extension of \cite[Proposition 2.11]{CaoHoering2019}.

\begin{lemma}
\label{lemma:num-flat-criterion-2}
Let $X$ be a projective manifold and $\cE$ a reflexive sheaf on $X$ with $c_1(\cE)=0$. 
Let $X_0 \subseteq X$  be an open subset whose complement has codimension at least $2$. 
Assume that $\cE$ is locally free over $X_0$.
Let $Y = \mathbb{P}(\cE|_{X_0})$ and $L$ the tautological line bundle.  
If the diminished base locus $\mathbb{B}_-(L)$ is not surjective onto $X_0$, then $\cE$ is a numerically flat vector bundle.
\end{lemma}

\begin{remark}\label{remark:base-locus}
We note that, in the previous lemma, even though $Y$ is only quasi-projective, we can define $\mathbb{B}_-(L)$ in the following way. 
Let $A$ be an ample divisor in $X$ and $f\colon Y \to X$ the natural morphism. 
Then we set 
\[\mathbb{B}_-(L) = \bigcup_{a>0}\bigcap_{m> 0}\mathrm{Bs}( L^{\otimes am} \otimes \cO_{Y}(f^*mA)).\]
Here, for any line bundle $M$ on $Y$,  $\mathrm{Bs}(M)$ is the base locus of $H^0(Y,M)$, regarded as a closed subset of $Y$.
Then we see that, the locus $\mathbb{B}_-(L)$ is not surjective onto $X_0$ if and only if   for any  integer  $a>0$, there exists an integer $m>0$, such that the natural evaluation map
\[H^0(X,{(\Sym}^{am} \cE)^{**} \otimes \cO_X(mA)) \to  ({\Sym}^{am} \cE)^{**}_x \otimes \cO_X(mA)_x \]
is surjective for general $x\in X$.
\end{remark}

\begin{proof}[{ Proof of Lemma \ref{lemma:num-flat-criterion-2} }]
We may assume that $\dim X \geqslant 2$. We may also assume that $X_0$ is the locally free locus of $\cE$. 
Since $\cE$ is reflexive, the complement of $X_0$ has codimension at least 3.
Let $H$ be an ample enough divisor in $X$ and $S$ the complete intersection surface of very general elements of the linear  system $|H|$.
 Then  $S \subseteq X_0$ and $\cE|_S$ is a vector bundle.
By Theorem \ref{thm:num-flat-criterion-1} and Mehta-Ramanathan theorem, it is enough to show that $\cE|_S$ is numerically flat.
The assumptions in the lemma imply that the diminished base locus of $\cO_{\mathbb{P}(\cE|_S)}(1)$ is not surjective onto $S$.
Hence $\cE|_S$ is numerically flat by \cite[Proposition 2.11]{CaoHoering2019}.
\end{proof}

As pointed out in the paragraph after \cite[Corollary 3.10]{Simpson1992},  Simpson proved that a numerically flat vector bundle is indeed flat.
In general, there may exist non isomorphic flat structures on a holomorphic vector bundle (see Example \ref{example:flat-structure} below). 
In the following Theorem \ref{thm:numflat=flat}, we underline a ``canonical'' flat structure on a numerically flat vector bundle.  
In the remainder of this paper, we will work with it for any numerically flat vector bundle. 
We note that, such a flat structure is exactly the one obtained from the correspondence in \cite[Section 3]{Simpson1992}. 

We recall some basic results about connections on complex vector bundles. 
Let  $\nabla$  be a smooth connection on a smooth complex vector bundle $\cQ$ on a complex manifold. 
If the  curvature of $\nabla^{0,1} $  is zero, 
then $\nabla$ determines a  holomorphic bundle structure $\cF$ on $\cQ$ such that $\nabla^{0,1}=\bar{\partial}_{\cF}$.  
Conversely, every holomorphic vector bundle $\cF$ admits a smooth connection such that $\nabla^{0,1}=\bar{\partial}_{\cF}$.  
For more details, see \cite[Section 1.3]{Kobayashi2014}.   

A  holomorphic  vector bundle   on a compact K\"ahler manifold  is called Hermitian flat if it admits a Hermitian metric  whose Chern connection is flat. 
In this case,  this connection is also the Hermitian-Yang-Mills connection, which is  unique.

Let $\cE$ be a numerically flat vector bundle   on a projective manifold $X$. 
Thanks to \cite[Theorem 1.18]{DemaillyPeternellSchneider1994},   there is a sequence of holomorphic  subbundles 
\[0=\cE_0 \subseteq  \cdots \subseteq \cE_k = \cE \]
such that each $\cG_i=\cE_i/\cE_{i-1}$ is  a  Hermitian flat vector bundle with the unique Hermitian flat connection $\nabla_i$.
Such a filtration always splits as smooth complex vector bundle. 
Hence if we set $\cQ=\bigoplus_{i=1}^k \cG_i$, then $\cQ$ is   isomorphic to $\cE$ as smooth vector bundles.
Thus in the following theorem,  to  construct a connection on $\cE$, we work with   $\cQ$.

\begin{thm}\label{thm:numflat=flat}
With the notation above,  there is a unique  flat  connection $\nabla_\cF$ on the complex  vector bundle  $\cQ$   such that 
\begin{enumerate}
\item  $\nabla_\cF$ induces a   holomorphic structure $\cF$ on $\cQ$, which is isomorphic to $\cE$ as holomorphic vector bundles;
\item  if $\{\cF_i\}$ is the filtration of holomorphic  subbundles on $\cF$, induced by $\{\cE_i\}$ and the isomorphism of item (1), then   $\{\cF_{i}\}$ is compatible with $\nabla_\cF$; 
\item  the $(1,0)$-part   $\nabla_\cF^{1,0}$ is    $\bigoplus_{i=1}^k \nabla_i^{1,0}$.
\end{enumerate}
\end{thm}

\begin{proof}
The existence of $\nabla_\cF$ goes back to \cite[Section 3]{Simpson1992}. 
We also refer to \cite[Section 3]{Deng2021} for a more explicit construction. 
For the uniqueness, we will prove by induction on the length $k$ of the filtration. 
If $k=1$, then the connection is the unique Hermitian flat connection.

Now we assume the uniqueness for lengths  smaller than $k$. Let  $\nabla_a$ and $\nabla_b$ be two connections on $\cQ$ which satisfy the conditions. 
By assumption,  $\nabla_a$ and $\nabla_b$ induce  isomorphic holomorphic structures. 
Such an isomorphism corresponds to some smooth automorphism  $\varphi$ of $\cQ$.
By induction, its  restrictions on $\bigoplus_{i=1}^{k-1} \cG_i$ and on $\bigoplus_{i=2}^k \cG_i$ are the identity maps. 
Therefore, item (2) of the theorem implies that  there is some smooth function $\eta$ with values  in $\cG_k^*\otimes \cG_1$ such that $\varphi$ is of the shape 
\[\varphi = \mathrm{Id}_{\cQ} + \eta,\]
where we view $\eta$ as a smooth endomorphism of $\cQ.$
Hence we have  
\[\nabla_b^{0,1} = \nabla_a^{0,1}+ \nabla_{k \to 1}^{0,1} \eta,\] 
where $\nabla_{k\to 1}$ is the Hermitian flat connection on $\cG_k^*\otimes \cG_1$ induced by $\nabla_1$ and $\nabla_k$.
Since $\nabla_a$ and $\nabla_b$ have the same $(1,0)$-part, it follows that 
$\nabla_b = \nabla_a+ \nabla_{k \to 1}^{0,1} \eta$.
We write $\delta = \nabla_{k \to 1}^{0,1} \eta$. 
The conditions on  flatness imply  that \[\nabla_{k \to 1} \delta =0.\]
Since $\nabla_{k \to 1}$ is a Hermitian flat connection, the $\partial \bar{\partial}$-lemma  implies that $\delta =0$. 
Hence $\nabla_a = \nabla_b$.
\end{proof}

\begin{remark}\label{remark:any-filtration}
We note that Theorem  \ref{thm:numflat=flat} holds for any filtration 
\[0=\cE_0 \subseteq  \cdots \subseteq \cE_k = \cE \]
on $\cE$ such that each graded piece $\cE_i/\cE_{i-1}$ is Hermitian flat. 
Indeed, we will prove later that the connection in the theorem is independent of the choice of the filtration, see Corollary \ref{cor:morphism-flat} and Remark \ref{remark:unique-connection} below.
\end{remark}

\begin{example}\label{example:flat-structure}
In general, a holomorphic vector bundle may have non isomorphic flat structures, all compatible with the holomorphic structure.  
Let $X$ be an elliptic curve and $\cE = \cO_X \oplus \cO_X$. We write $\mathbf{e}_1$ and $\mathbf{e}_2$ for the generators of the two summand of $\cE$.
We note that $\Omega_X^1\cong \cO_X$. For any constant holomorphic 1-form $s\in H^0(X, \Omega_X^1)$, we can define a   flat connection $\nabla_s$ on $\cE$ as follows,
\[ \nabla_s (f\mathbf{e}_1 + g\mathbf{e}_2) = \mathrm{d}f \mathbf{e}_1 +\mathrm{d}g\mathbf{e}_2 + gs\mathbf{e}_1,\]
where $f,g$ are arbitrary smooth functions on $X$.
Then $(\cE,\nabla_s)$ is isomorphic to $(\cE,\nabla_0)$ as   flat vector bundles if and only if $s=0$.
Furthermore, the unique connection satisfying the conditions of Theorem \ref{thm:numflat=flat} is $\nabla_0$.
\end{example}

Until the end of this paper, a connection on a holomorphic vector bundle is always assumed to be compatible with the holomorphic structure.  
We say that a flat connection on a numerically flat vector bundle $\cE$ satisfies the conditions of Theorem \ref{thm:numflat=flat} if it is induced by the isomorphism $\cE \cong \cF$ of Theorem \ref{thm:numflat=flat}.


\begin{lemma}\label{lemma:section-flat}
Let $\cE$ be a numerically flat vector bundle on a projective manifold $X$. 
Let $\nabla_\cE$ be  a flat connection on $\cE$  which satisfies the conditions of Theorem \ref{thm:numflat=flat}. 
Then any section $s \in H^0(X,\cE)$ is parallel with respect to $\nabla_\cE$.
\end{lemma}

\begin{proof}
We first note that $\nabla_\cE^{0,1}s=0$ as $s$ is holomorphic. It remains to show that $\nabla_\cE^{1,0}s = 0$.
There is a sequence of subbundles 
\[0=\cE_0 \subseteq  \cdots \subseteq \cE_k = \cE,\] 
such that each $\cG_i=\cE_i/\cE_{i-1}$ is  an irreducible Hermitian flat vector bundle with connection $\nabla_i$.
We will prove by induction on the length  $k$.
If $\cE$ is irreducible Hermitian flat, then it is slope stable. Hence  either $\cE \cong \cO_X$ or $s=0$. In both cases,  $s$ is constant and $\nabla_\cE s=0$. 

Now we assume the statement is true for lengths smaller than $k$. We have the exact sequence 
\[0\to \cE_{k-1} \to \cE \to \cG_k \to 0.\]
If   $s \in H^0(X, \cE_{k-1})$ then  $\nabla_\cE^{1,0} s= 0$ by induction.
Assume that $s\notin H^0(X, \cE_{k-1})$. 
Then $\cG_k\cong \cO_X$ as it is an irreducible Hermitian flat vector bundle which admits a nonzero global holomorphic section.
In this case,  the exact sequence above splits and  $\cE \cong \cE_{k-1} \oplus \cG_k$ as holomorphic vector bundles.  
We can decompose $s=s'+s''$ according to this direct sum. 
From the item (3) of Theorem \ref{thm:numflat=flat}, we see that
\[\nabla_\cE^{1,0} = (\nabla_\cE|_{\cE_{k-1}})^{1,0} \oplus \nabla_k^{1,0}, \] 
where $\nabla_k$ is the unique Hermitian flat connection on $\cG_k \cong \cO_X$. 
Then we have $\nabla_k^{1,0} s''= 0$, and  $(\nabla_{\cE}|_{\cE_{k-1}})^{1,0}(s')=0$ by induction. Hence $\nabla_\cE^{1,0} s = 0$.
\end{proof}

\begin{remark}
It is indispensable to specific the connection in Lemma \ref{lemma:section-flat}. 
For instance, we consider a  flat vector bundle $(\cE,\nabla_s)$ on an elliptic curve $X$ defined in Example \ref{example:flat-structure}, such that $s$ is a nonzero constant holomorphic 1-form. 
The generator $\mathbf{e}_2$ can be view as a section $\mathbf{e}_2 \in H^0(X,\cE)$. 
Then $\nabla_s \mathbf{e}_2 = s \mathbf{e}_1$ is not zero. That is, $\mathbf{e}_2$ is not parallel with respect to $\nabla_s$.
\end{remark}

\begin{cor}\label{cor:morphism-flat}
Let $\varphi \colon \cE \to \cG$ be a generically surjective morphism between numerically flat vector bundles on a projective manifold $X$.
Then $\varphi$ is surjective. 

Assume  further that   $\cE$ and $\cG$ are equipped with   flat connections  $\nabla_\cE$ and $\nabla_\cG$ respectively,  which satisfy the conditions of Theorem \ref{thm:numflat=flat}. 
Then  $\varphi$  is a morphism of flat vector bundles.
\end{cor}

\begin{proof}
To prove that $\varphi$ is surjective, it is enough to show that the induced morphism 
\[\cO_X \to (\det \cG) \otimes (\bigwedge^l \cE^*)\] does not vanish at any point of $X$.
Here $l$ is the rank of $\cG$.
Since $(\det \cG^*) \otimes (\bigwedge^l \cE)$ is numerically flat, hence nef, the statement   follows from \cite[Lemma 1.16]{DemaillyPeternellSchneider1994}.

For the second part of the corollary, we may identify $\varphi$   as an element of  $H^0(X,\cE^*\otimes \cG)$. 
If 
\[ 0=\cE_0 \subseteq  \cdots \subseteq \cE_k = \cE^*  \mbox{ and }  0=\cF_0 \subseteq  \cdots \subseteq \cF_l = \cF\]
are filtrations whose graded pieces are Hermitian flat vector bundles, 
then  we have a filtration 
\[ 0=\cG_0  \subseteq  \cdots \subseteq \cG_t = \cE^* \otimes \cF \]
such that 
\[\cG_s = \bigoplus_{i+j=s} (\cE_i\otimes \cF_j).\]  
Furthermore, for each $s$ we have, 
\[
\cG_{s}/\cG_{s-1}  \cong \bigoplus_{i+j=s} (\cE_i/\cE_{i-1}\otimes \cF_j/\cF_{j-1}).
\]
As a result, the graded pieces of the filtration $\{\cG_s\}$ are Hermitian flat vector bundles. 
We can then verify that the tensor connection $\nabla_{\cE^*\otimes \cG}$ on $\cE^*\otimes \cG$ satisfies the conditions of 
Theorem \ref{thm:numflat=flat} as well.  
Hence Lemma \ref{lemma:section-flat} implies that 
$\varphi$ is parallel with respect to $\nabla_{\cE^*\otimes \cG}$.  
It follows that, for any local smooth section $\sigma$ of $\cE$, we have 
\[ \nabla_\cG (\varphi(\sigma)) 
= \varphi(\nabla_\cE \sigma) + (\nabla_{\cE^*\otimes \cG}) \varphi (\sigma)  
= \varphi(\nabla_\cE \sigma).
\]
In another word,  $\varphi$  is a morphism of flat vector bundles.
\end{proof}

\begin{remark}\label{remark:unique-connection}
If we take $\varphi$  as the identity endomorphism of $\cE$ in the previous corollary, then we may deduce that a connection satisfying the conditions of Theorem \ref{thm:numflat=flat} is unique, independent of the choice of the filtration $\{\cE_i\}$.
\end{remark}



The   lemma  below reveals a relation between locally trivial families and flat  holomorphic vector bundles. 
 
\begin{lemma}\label{lemma:loc-trivial} 
Let $(\cE, \nabla_\cE)$ be a flat   vector bundle on a complex manifold $X$. 
Assume   that  there is a  surjective morphism of  graded commutative $\cO_X$-algebras   
$$\bigoplus_{p\geqslant 0} {\Sym}^p \cE \to \bigoplus_{p\geqslant 0} \cQ_p,$$ 
such  that $\cQ_0 = \cO_X$ and  that each graded piece  ${\Sym}^{p}\cE \to \cQ_p$  is a  morphism of   flat   vector bundles. 
Then 
$$f\colon Z= \mathrm{Proj}_{\cO_X} \bigoplus_{p\geqslant 0} \cQ_p \to X$$ is a locally trivial family over $X$.
Moreover, the connection $\nabla_{\cE}$ induces a foliation $\cG$ on $Z$ such that $T_Z = \cG \oplus \cF$, where $\cF$ is the foliation induced by $f$.
\end{lemma}

\begin{proof}
We consider  the following  commutative diagram, where the vertical arrows are  product maps  
\begin{equation*} 
\begin{tikzcd}[column sep=large, row sep=large]
    {\Sym}^i \cE \otimes {\Sym}^j \cE   \ar[d," "]      \ar[r," "]  & \cQ_i\otimes \cQ_j  \ar[d," "]   \\
  {\Sym}^{i+j} \cE \ar[r," "] &   \cQ_{i+j}
\end{tikzcd}
\end{equation*} 
Since the horizontal maps and the first vertical map are   morphisms of flat vector bundles, and since the horizontal maps are surjective,  we see  that   $\cQ_i\otimes \cQ_j \to \cQ_{i+j}$
is a  morphism  of flat vector bundles as well.

Let $\pi\colon \widetilde{X} \to X$ be the universal cover of $X$. 
Then the connection $\nabla_{\cE}$ determines   isomorphisms  
$$\pi^*\cE \cong E \otimes \cO_{\widetilde{X}}  \mbox{ and } \pi^*\cQ_p \cong Q_p \otimes \cO_{\widetilde{X}},$$
where $E$ and $Q_p$ are complex vector spaces equipped with $\pi_1(X)$-actions.
Moreover, under these isomorphisms, the natural morphism 
$$\bigoplus_{p\geqslant 0} {\Sym}^{p}\pi^*\cE \to \bigoplus_{p\geqslant 0} \pi^*\cQ_p$$ 
is  induced by a surjective $\pi_1(X)$-equivariant morphism of  graded commutative $\mathbb{C}$-algebras   
$$\bigoplus_{p\geqslant 0} {\Sym}^p  E \to \bigoplus_{p\geqslant 0}  Q_p.$$

Let  $I$ be the kernel of the previous morphism and  $$F = \mathrm{Proj} \bigoplus_{p\geqslant 0} Q_p.$$ 
Then  $F$ is the subvariety of  $\mathbb{P}(E)$ with graded ideal $I$.  
Since $I$ is stable under the $\pi_1(X)$-action, we see that $F$ is stable under the natural $\pi_1(X)$-action on $\mathbb{P}(E)$. 
Hence  we  obtain  an induced   polarized $\pi_1(X)$-action on $F$.
Moreover, there is a $\pi_1(X)$-equivariant isomorphism $Z\times_X \widetilde{X} \cong F\times \widetilde{X}$. 
This shows that $Z$ is locally trivial.
The relative tangent bundle $\widetilde{\cG}$ of the  natural projection  $F\times \widetilde{X} \to F$ is $\pi_1(X)$-equivariant. 
Hence it descends to a foliation $\cG$ on $Z$.  Furthermore, we have  $T_Z = \cG \oplus \cF$.
\end{proof}

\begin{remark}\label{remark:Ehresmann-connection} 
The foliation $\cG$ in the previous lemma has the following alternative description. 
The flat connection $\nabla_\cE$ induces a flat Ehresmann connection on $\mathbb{P}(\cE)$, which can be viewed as a foliation $\cH$ on $\mathbb{P}(\cE)$. 
By construction, we have a  natural embedding $Z \to \mathbb{P}(\cE)$ of  fiber bundles over $X$. 
Then $\cH$ induces an  Ehresmann connection  on $Z$, which is exactly $\cG$. 
\end{remark}

The following proposition on  representations of fundamental groups is an application of \cite[Theorem 4.8]{Kollar1993}.

\begin{prop}\label{prop:flat-descend}
Let $f\colon X\to Y$ be a proper fibration between  smooth complex algebraic varieties. 
Assume that for every prime divisor $B$ in $Y$, there is an irreducible component of $f^*B$ which is reduced and dominates $B$. 
Let  $\xi$ be a linear representation of the fundamental group $\pi_1(X)$.
If $\xi$ induces trivial representation on general fibers of $f$, then $\xi$ factors through $\pi_1(Y)$.
\end{prop}

\begin{proof}
By removing some closed subset of codimension at least $2$ in $Y$, we may assume that $f$ is equidimensional. 
Let $F$ be a general fiber of $f$, and $H \subseteq \pi_1^{}(X)$ the intersection of all normal subgroups of finite index which contain the image of $\pi_1(F)$. 
Since a linear group is residually finite, and since   $\mathrm{Ker}(\xi)$ contains the image of $\pi_1(F)$, we obtain that  
$H \subseteq \mathrm{Ker}(\xi)$.

We write $K$ the kernel of the surjective morphism  $\theta \colon \pi_1^{}(X) \to \pi_1^{}(Y)$. The proposition is equivalent to show that $$K \subseteq \mathrm{Ker}(\xi).$$ From the previous discussion, it is enough to show that $K \subseteq H$.
Let $G\subseteq \pi^{}_1(X)$ be a normal subgroup of finite index  which contains $H$. 
We need to show that it contains $K$.
Let $p\colon X(G) \to X$ be the finite \'etale cover corresponding  to $G$, and $\pi\colon Y(G)\to Y$ the normalization of $Y$ in function field of $X(G)$. 
Since $G$ contains the image of $\pi_1(F)$, the morphism $p$ is a trivial cover over $F$.  

By assumption,  for any prime divisor $B$ in $Y$,  the pullback  $f^*B$ has a reduced component.  
It follows from \cite[Lemma 4.8.4]{Kollar1993} that $\pi$ is \'etale (we note that the proof of \cite[Lemma 4.8.4]{Kollar1993} does not require $Y$ to be proper). 
Then $\pi$ corresponds to a   subgroup $M \subseteq \pi_1^{}(Y)$ of finite index .  
Since $p$ is a trivial cover over $F$, it follows that $X(G)\cong X\times_Y Y(G)$.
Therefore $G = \theta^{-1}(M)$, and hence it  contains $K$.
This completes the proof.  
\end{proof}

We also need the following lemma for the proofs of the main theorems.

\begin{lemma}\label{lemma:iso-descend}
Let $f\colon X\to Y$ be a surjective morphism between smooth  quasi-projective varieties.  Assume that its general  fibers are proper connected.
Let $\cV$ and $\cW$ be two  vector bundles on $Y$. If there is an isomorphism  $\varphi\colon f^*\cV \cong f^*\cW$ on $X$, then $\varphi$ descends to an isomorphism  $\eta\colon \cV\cong \cW$ on $Y$.
\end{lemma}

\begin{proof}
We only need to show that $\varphi$ descends to a morphism $\eta \colon \cV\to \cW$ on $Y$. 
Indeed,  since $f\colon X\to Y$ is surjective, if $\eta$ exists, then it must be  an isomorphism.

The  morphism $\varphi$ can be viewed as an element in $H^0(X, f^*(\cV^* \otimes \cW))$.
By the  projection formula,  we have 
\[ f_*f^*(\cV^* \otimes \cW) \cong \cV^*\otimes \cW \otimes f_*\cO_X. \]
Hence it is enough to show that $f_*\cO_X\cong \cO_Y$. 

Let $V\subseteq Y$ be any non empty open subset and $U=f^{-1}(V)$.
We consider a holomorphic function $\sigma$ on $U$.
By assumption, there is an open dense subset $V_0\subseteq V$ such that $f$ has proper connected fibers over $V_0$. 
Hence $\sigma|_{U_0}$  descends to a holomorphic function $\mu_0$ on $V_0$, where $U_0=f^{-1}(V_0).$ 
We have to show that $\mu_0$ can be extended to a holomorphic function $\mu$ on $V$. 
By Riemann's Removable Singularity theorem, it is enough to show that $\mu_0$  is locally bounded over $V$. 
We also note that, by  Hartogs' theorem, we may shrink $V$ and assume that $f$ is equidimensional.

Let $y\in V$ be a point. Let $Z\subseteq X$ be the complete intersection of general very ample divisors such that $\dim Z = \dim Y$. 
By shrinking $V$ around $y$ if necessary, we may assume that  $Z\cap U$ is quasi-finite over $V$. 
Then the   morphism $$f|_{Z\cap U} \colon Z\cap U \to V$$ is an open morphism, with respect to the analytic Euclidean topology. 
Let $x\in Z\cap U$ be a point lying over $y$, and $U' \subseteq Z\cap U$ an Euclidean open neighborhood of $x$. 
Then $V' = f(U')$ is an  Euclidean open subset  of $V$ containing $y$. 
Since $\sigma$ is holomorphic, we may assume that $\sigma|_{U'}$ is bounded. 
Hence $\mu_0|_{V_0 \cap V'}$ is bounded. That is, $\mu_0$ is   bounded around $y$. 
This completes the proof.
\end{proof}

\section{Algebraically integrable foliations with semistable leaves and nef anticanonical classes}\label{section:semistable}

Algebraically integrable foliations with nef anticanonical classes induced by rational maps which are semistable in codimension one are well studied in \cite{Cao19}, \cite{CaoHoering2019} and \cite{CCM19}. 
Their results are crucial for the proofs of the main theorems in the present paper.
For the reader's convenience, we summarize  some of them here. 


Throughout this section, we consider the following situation. The morphism $\varphi\colon \Gamma \to Y$ is a fibration between projective manifold, and $\pi$ is a birational morphism onto a normal projective variety $X$. 
\begin{equation*} 
\begin{tikzcd}[column sep=large, row sep=large]
  \Gamma \ar[d,"{\varphi}"]      \ar[r,"\pi"]  & X   \\
 Y &  
\end{tikzcd}
\end{equation*}
We make the following  assumptions.
\begin{enumerate}
\item The $\pi$-exceptional locus $E$ is pure of codimension one, and it does not dominate the base $Y$. 
\item Every $\varphi$-exceptional divisor is $\pi$-exceptional.
\item There is a $\pi$-exceptional $\mathbb{Q}$-divisor $E''$ such that 
$$-K_{\Gamma/Y}+ E''$$
is   nef.
\item For every prime divisor $B\subseteq Y$, any non reduced component of $\varphi^*B$ is $\pi$-exceptional. 
\item There is a finite group $G$ acting on the varieties $\Gamma$, $Y$ and $X$, such that the morphisms $\varphi$ and $\pi$ are $G$-equivariant. 
Furthermore, the quotient $X/G$ is $\mathbb{Q}$-factorial.
\end{enumerate}

We recall that a prime divisor in $\Gamma$ is $\varphi$-exceptional if its image has codimension at least two in $Y$.  
If $E''=\sum a_iE_i$ is the decomposition into prime divisors, then we set  $|E''| = \sum |a_i|E_i$. 
The main objective of this section is to prove the following proposition.

\begin{prop}\label{prop:irred-fiber-semistable}
With the assumption above, for every prime divisor $B\subseteq Y$, there is at most one irreducible component of $\varphi^{-1}(B)$ which is not contained in the $\pi$-exceptional locus.
\end{prop}


The proof of the proposition will be postponed to the end of the section.   
We fix  a sufficiently ample, $G$-invariant   divisor $A$ in $\Gamma$, so that it is $\varphi$-relatively very ample, and that for each $p\geqslant 1$, the natural morphism 
\[ \mathrm{Sym}^p \varphi_{*} \cO_\Gamma(A) \to \varphi_* \cO_{\Gamma}(pA)  \]
is  surjective. 
Let $Y_0 \subseteq Y$ be a Zariski open subset such that the following properties hold.
\begin{enumerate}
\item[$\bullet$] The fibration $\varphi$ is equidimensional over $Y_0$.
\item[$\bullet$] For any prime divisor $B\subseteq Y$, the preimage $\varphi^{-1}(B)$ is   contained in the $\pi$-exceptional locus if and only if $B\subseteq Y\setminus Y_0$. 
\end{enumerate}
 
\begin{remark}
In \cite[Section 3]{CCM19}, $Y_0$ is defined as the largest open subset which satisfies the properties. However, for later use in this paper, we need to remove from $Y_0$ some closed subset of codimension at least two. Therefore we define $Y_0$ in this way.
\end{remark}

\begin{lemma}[{\cite[Lemma 3.5]{CCM19} }]
\label{lemma:positivity-direct-image}
Let $L$ be a $\varphi$-relatively big Cartier divisor in $\Gamma$ and $P$ a $\mathbb{Q}$-divisor in $Y$. Assume that $L-\varphi^*P$ is pseudoeffective. Then for any large enough integer $c$, 
\[\varphi_*\cO_{\Gamma}(L+cE'') - P\] is weakly positively curved if it is not zero. 
That is, for a fixed K\"ahler form $\omega_Y$ on Y, for a real  number $\epsilon>0$, there is a possibly singular Hermitian metric $h_\epsilon$ on $\varphi_*\cO_{\Gamma}(L+cE'')$ such that 
 \[  \sqrt{-1}\Theta_{h_\epsilon }(\varphi_*\cO_{\Gamma} (L+ cE''))  \succeq (\beta - \epsilon \omega_Y) \otimes \mathrm{Id}_{\varphi_*\cO_{\Gamma} (L + cE'')}\] over the locally free locus of $\varphi_*\cO_{\Gamma}(L+cE'')$. 
 Here $\beta$ is a smooth closed $(1,1)$-form representing $P$.
\end{lemma}

\begin{lemma}
\label{lemma:normalized-c_1}
Let $L$ be a $\varphi$-relatively big, $G$-invariant   divisor in $\Gamma$, and let $m>0$ be an integer. 
Then the $\mathbb{Q}$-divisor 
\[
L-\frac{1}{r}\varphi^*c_1(\varphi_*\cO_\Gamma(L+mE))
\] 
is the sum of a pseudoeffective divisor and a  $\pi$-exceptional divisor. 
Here $r$ is the rank of   $\varphi_*\cO_\Gamma(L+mE)$.
\end{lemma}


\begin{proof}
Let $A_Y$ be a $G$-invariant ample divisor in $Y$.
At the end of the proof of {\cite[Lemma 3.6]{CCM19}}, it was shown that,  for any integer $p >0$, there is an effective $\pi$-exceptional divisor $F_p$ such that the $\mathbb{Q}$-divisor
\[  
 L- \frac{1}{r}\varphi^*c_1(\varphi_*\cO_\Gamma(L+mE)) + \frac{1}{rp}\varphi^*A_Y + F_p
\]
is  $\mathbb{Q}$-linearly equivalent to an effective divisor $H_p$. 
Replacing   $F_p$ by the sum of its $G$-orbits, we may assume that  $F_p$ is $G$-invariant.  
Since $L$ and $E$ are $G$-invariant, and since $\varphi$ is $G$-equivariant, we see that, for any element $g\in G$, there is a linear equivalence 
\[ c_1(\varphi_*\cO_\Gamma(L+mE)) \sim  g^* c_1(\varphi_*\cO_\Gamma(L+mE)). \]
Hence,  replacing  $c_1(\varphi_*\cO_\Gamma(L+mE))$ by 
\[
\frac{1}{|G|} \sum_{g\in G} g^*c_1(\varphi_*\cO_\Gamma(L+mE)),
\] 
we may assume that it is represented by a $G$-invariant $\mathbb{Q}$-divisor.  
It then follows that 
$ H_p \sim_\mathbb{Q}  g^* H_p$ for any $g\in G$. 
Replacing $H_p$ by $\frac{1}{|G|} \sum_{g\in G} g^*H_p$, we may assume that  it is $G$-invariant as well.

Let $\rho\colon \Gamma \to X/G$  and $q\colon X \to X/G$  be the natural morphisms. 
We  see   that 
\begin{eqnarray*}
&& \pi_*( L- \frac{1}{r} \varphi^*c_1(\varphi_*\cO_\Gamma(L+mE))   + \frac{1}{rp} F_p) \\
&=& 
\pi_*( L- \frac{1}{r} \varphi^*c_1(\varphi_*\cO_\Gamma(L+mE)) )
\end{eqnarray*}
is a $G$-invariant $\mathbb{Q}$-divisor in $X$. 
Hence there are $\mathbb{Q}$-divisors  $D$, $A'_Y$ and $H'_p$ in $X/G$ such that  
$q^*H'_p = \pi_*H_p$, that $q^*A'_Y = \pi_* (\varphi^*A_Y)$ and that 
\[
q^*D  = \pi_*( L- \frac{1}{r} \varphi^*c_1(\varphi_*\cO_\Gamma(L+mE)) ).
\]
It follows that
\[ 
 D + \frac{1}{rp}A'_Y \sim_\mathbb{Q} H'_p \geqslant 0. 
\]
By tending $p$ to  the infinity, we deduce that   $D$ is pseudoeffective.  
Since $X/G$ is $\mathbb{Q}$-factorial, we conclude that there is some $\pi$-exceptional divisor $F$ such that 
\[
 L- \frac{1}{r} \varphi^*c_1(\varphi_*\cO_\Gamma(L+mE)) = \rho^*D + F.
\]
This completes the proof of the lemma.
\end{proof}

\begin{lemma}
\label{lemma:A-tilde}
There is some positive integer $m_0$ such that for  any integer $m\geqslant m_0$, the torsion-free sheaf
$\varphi_*\cO_\Gamma (A+mE)$ has the same rank $r_{m_0}$.
Moreover, for any effective $\pi$-exceptional divisor $\widetilde{E}$,  
 the natural morphism
\[ \det(\varphi_*\cO_\Gamma (A+m_0 E)) \to \det(\varphi_*\cO_\Gamma (A+m_0E + \widetilde{E}))\]
is an isomorphism over $Y_0.$
\end{lemma}
 
\begin{proof}
The fact that the ranks are constant follows from the assumption that $E$ does not dominates $Y$. 
For the second assertion, it is enough to apply item (ii) of {\cite[Proposition 3.7]{CCM19}} by letting $p=1$.  
\end{proof}

\begin{lemma}
 \label{lemma:symA}
We set  $M= A+m_0 E$. 
Let $p>0$ be an integer and  $s$ the rank of  $\varphi_*\cO_{\Gamma}(pM))$. 
Then for any effective $\pi$-exceptional divisor $\widetilde{E}$,  
$$ \frac{1}{r_{m_0}} c_1(\varphi_* \cO_{\Gamma}(M)) - \frac{1}{ps}c_1(\varphi_* \cO_{\Gamma}(pM + \widetilde{E}))$$ 
is the sum of a pseudoeffective divisor and a divisor supported in $Y\setminus Y_0$.
\end{lemma}

\begin{proof}
By replacing $A$ by $pA$ in Lemma \ref{lemma:A-tilde}, we deduce that there is an integer $k>0$ such that if $m\geqslant k$, then the morphism
\[ \det(\varphi_*\cO_\Gamma (pM+kE)) \to \det(\varphi_*\cO_\Gamma (pM+mE))\]
is an isomorphism  over $Y_0$.  

Since $E$ is  $\pi$-exceptional, it does not dominate $Y$ by assumption. 
Therefore $\varphi_*\cO_\Gamma (pM + pkE)$ has the same rank as $\varphi_*\cO_\Gamma (pM)$, which is  $s$. 
By applying Lemma \ref{lemma:normalized-c_1} (with $L=pM$), there is some integral  effective $\pi$-exceptional divisor $Q$ such that the $\mathbb{Q}$-divisor 
\[M + kE + Q  - \frac{1}{ps} \varphi^*c_1(\varphi_*\cO_\Gamma (pM + pkE))\] 
 is pseudoeffective. 
Hence by  Lemma \ref{lemma:positivity-direct-image}, there is some integer $c>0$ such that 
$$\varphi_*\cO_{\Gamma}(M + kE + Q + cE'') - \frac{1}{ps} c_1(\varphi_*\cO_\Gamma (pM + pkE))$$
is weakly positively curved. 
That is, for a fixed K\"ahler form $\omega$ on $Y$, for each $\epsilon > 0$, there is some possibly singular Hermitian metric $h_\epsilon$ on $\varphi_*\cO_{\Gamma}(M + kE + Q + cE'')$
such that \[  \sqrt{-1}\Theta_{h_\epsilon }(\varphi_*\cO_{\Gamma} (M + kE + Q + cE''))  \succeq (\beta - \epsilon \omega_Y) \otimes \mathrm{Id}_{\varphi_*\cO_{\Gamma} (M + kE + Q + cE'')}\]
over the locally free locus of $\varphi_*\cO_{\Gamma} (M + kE + Q + cE'')$, 
where $\beta$ is a smooth $(1,1)$-from representing $\frac{1}{ps} c_1(\varphi_*\cO_\Gamma (pM + pkE))$.
By taking the determinant, we obtain that 
\[ c_1(\varphi_*\cO_{\Gamma} (M + kE + Q + cE'')) - \frac{r_{m_{0}}}{ps} c_1(\varphi_*\cO_\Gamma (pM + pkE))\]
is pseudoeffective.
Then \[ c_1(\varphi_*\cO_{\Gamma} (M + kE + Q + c|E''|)) - \frac{r_{m_0}}{ps} c_1(\varphi_*\cO_\Gamma (pM + pkE))\]
is pseudoeffective as well.

By Lemma \ref{lemma:A-tilde}, we see that $c_1(\varphi_*\cO_{\Gamma} (M + kE + Q + c|E''|)) - c_1(\varphi_*\cO_{\Gamma} (M))$ is supported in $Y\setminus Y_0$. 
Hence \[ c_1(\varphi_*\cO_{\Gamma} (M)) - \frac{r_{m_0}}{ps} c_1(\varphi_*\cO_\Gamma (pM + pkE))\] is the sum of a pseudoeffective divisor and a divisor supported in $Y\setminus Y_0$.

Finally, if  $\widetilde{E}$ is an arbitrary effective $\pi$-exceptional divisor, then there is some integer $k'\geqslant k$ such that $\widetilde{E}\leqslant pk'E.$ 
It follows that $$c_1(\varphi_*\cO_\Gamma (pM + pk'E)) - c_1(\varphi_*\cO_\Gamma (pM + \widetilde{E}))$$ is effective. 
From the assumption of the first paragraph, we see that 
\[c_1(\varphi_*\cO_\Gamma (pM + pkE)) - c_1(\varphi_*\cO_\Gamma (pM + pk'E))\] 
is supported in $Y\setminus Y_0$. 
Thus we deduce that 
$$ c_1(\varphi_* \cO_{\Gamma}(M)) - \frac{r_{m_0}}{ps}c_1(\varphi_* \cO_{\Gamma}(pM + \widetilde{E}))$$ 
is the sum of a pseudoeffective divisor and a divisor supported in $Y\setminus Y_0$. 
\end{proof}

We define the following $\mathbb{Q}$-divisor on $\Gamma$ 
\[\widetilde{A} = A+m_0E - \frac{1}{r_{m_0}}\varphi^*c_1(\varphi_*\cO_\Gamma (A+m_0E)) = M-\frac{1}{r_{m_0}}\varphi^*c_1(\varphi_*\cO_\Gamma (M)).\] 
By Lemma \ref{lemma:normalized-c_1}, there is some integral $\pi$-exceptional effective divisor $F$ such that $\widetilde{A}+F$ is pseudoeffective.

\begin{lemma}\label{lemma:V}
With the notation above, let $p> 0$ be an integer. 
Then   the direct image 
\[\cV = \varphi_*\cO_\Gamma (c|E''|+pr_{m_0}(\widetilde{A}+F))\]
is  weakly positively curved for any integer $c$ large enough. 
Moreover, $c_1((\varphi^*\cV)^{**})$ is supported in the $\pi$-exceptional locus, in the following sense. 
For any ample  divisors $D_1,...,D_{n-1}$ on $X$, we have $$c_1((\varphi^*\cV)^{**})\cdot \pi^*D_1 \cdots \pi^*D_{n-1} =0,$$ where $n$ is the dimension of $\Gamma$. 
\end{lemma}

\begin{proof}
Since $\widetilde{A}+F$ is pseudoeffective, by  Lemma \ref{lemma:positivity-direct-image}, the direct image  
$$\varphi_*\cO_\Gamma (cE''+pr_{m_0}(\widetilde{A}+F))$$ 
is weakly positively curved for any large enough integer $c$. 
We note that there is a natural injective morphism   
$$\varphi_*\cO_\Gamma (cE''+pr_{m_0}(\widetilde{A}+F)) \to \varphi_*\cO_\Gamma (c|E''|+pr_{m_0}(\widetilde{A}+F)),$$ 
which is generically an isomorphism.
Hence $\cV$  is weakly positively curved as well by \cite[Proposition 2.5]{CaoHoering2019}. 

To finish the proof, we fix some $c$ so that $\cV$ is weakly positively curved. 
Then $c_1( \cV)$ is pseudoeffective.  
By definition,   we have
\[
c|E''|+pr_{m_0}(\widetilde{A}+F) = c|E''| + pr_{m_0}F + pr_{m_0} M - p\varphi^*c_1(\varphi_*(\cO_\Gamma(M) )). 
\]
By applying the projection formula, this  implies that 
\[-c_1(\cV) = sp c_1(\varphi_*\cO_{\Gamma}(M)) -c_1(\varphi_*\cO_{\Gamma}(c|E''|+pr_{m_0}F +pr_{m_0}M)),\]
where $s$ is the rank of $\cV$. 
By Lemma \ref{lemma:symA},  we obtain that
 $-c_1(\cV)$ is the sum of a pseudoeffective divisor and some divisor supported in $Y\setminus Y_0$. 
Since every $\varphi$-exceptional divisor is  also $\pi$-exceptional, and since every divisor in $\Gamma$ lying over $Y\setminus Y_0$ is $\pi$-exceptional,  it follows that   both 
$c_1((\varphi^*\cV)^{**})$  and $-c_1((\varphi^*\cV)^{**})$ can be written as the sum of a pseudoeffective divisor and a $\pi$-exceptional divisor.
Therefore $$c_1((\varphi^*\cV)^{**})\cdot \pi^*D_1 \cdots \pi^*D_{n-1} =0$$  for any ample  divisors $D_1,...,D_{n-1}$ on $X$. 
\end{proof}

We can then deduce the following corollary. 
The additional condition $c_i+c_j \leqslant c_{i+j}$ can be obtained by induction on $p$.

\begin{cor}\label{cor:Vp}
We set $c_0=0$. There is a sequence $\{c_p\}_{p> 0}$ of nonnegative integers such that  the direct image sheaves 
\[\mathcal{V}_{p} = \varphi_*\cO_\Gamma (c_{p}|E''|+pr_{m_0}(\widetilde{A}+F))\]
are weakly positively curved.
Moreover, $c_1((\varphi^*\cV_p)^{**})$ is supported in the $\pi$-exceptional locus, and $c_i+c_j \leqslant c_{i+j}$ for any $i,j \geqslant  0$. 
\end{cor}

\begin{remark}
\label{remark:linear-system}
The sheaves $\cV_p$ carry geometric properties of $\Gamma$ as follows.  
Let 
\[
L_p=c_p|E''|+pr_{m_0}(\widetilde{A}+F)
\] 
for all $p\geqslant 0$.
Then  $\cV_p$ can be viewed as the complete $\varphi$-relative linear system of the divisor $L_p$.
In particular, $\cV_1$ induces a rational map $g\colon \Gamma \dashrightarrow \mathbb{P}(\cV_1)$. 
Since $\widetilde{A}$ is sufficiently ample on general fiber of $\varphi$, and since neither $E''$ nor $F$ dominates $Y$, we see that $L_1$ is sufficiently ample on general fibers of $\varphi$. 
It follows that $g$ is a birational map to its image. We denote by $Z$ the (closure of the) image of $g$.
We have the following morphisms of graded $\cO_Y$-algebras
\[\bigoplus_{p\geqslant 0}{\Sym}^p \cV_1 \to \bigoplus_{p\geqslant 0} \varphi_*\cO_\Gamma (pL_1)  \to  \bigoplus_{p\geqslant 0} \cV_p.\]
If $\cA$ is the image of the first arrow above, then $Z \cong \mathrm{Proj}_{\cO_Y}\cA$.
In general,  it is   not trivial to compute $\cA$ and hence $Z$. 
In the situations considered in this paper, we manage to prove that  $\mathcal{A} =  \bigoplus_{p\geqslant 0} \mathcal{V}_p$  over some open dense subset of $Y$.
\end{remark}

\begin{cor}\label{cor:smooth-morphism}
Assume  that $\Gamma = X$ and that $\varphi$ is equidimensional. Then $X$ is   a locally trivial family over $Y$.
\end{cor}

\begin{proof} 
Under the assumption,  $F$ and  $E''$ are all zero. 
Let  $L_p=pr_{m_0}(\widetilde{A})$  for all $p\geqslant 0$.
Then  $L_1$ is $\varphi$-relatively very ample, $L_p = pL_1$, and $\Gamma \cong  \mathrm{Proj}_{\cO_{Y}} \bigoplus_{p\geqslant 0} \cV_p$. 
Moreover, thanks to Corollary \ref{cor:Vp},  each $\cV_p$ is reflexive, weakly positively curved, and  with zero first Chern class. Hence they are numerically flat vector bundles by  \cite[Proposition 2.6]{CCM19}. 
We note that  the morphism 
$ {\Sym}^p \cV_1 \to   \cV_p $
is   surjective for all $p$ since $A$ is sufficiently ample.  
Thus  $\Gamma    = \mathrm{Proj}_{\cO_{Y}} \bigoplus_{p\geqslant 0} \cV_p$ 
  is a locally trivial family after  Lemma \ref{lemma:loc-trivial}.
This completes the proof of the corollary.
\end{proof}

Now we will prove Proposition \ref{prop:irred-fiber-semistable}.

\begin{proof}[{Proof of Proposition \ref{prop:irred-fiber-semistable} }]

We write $L_p=c_p|E''|+pr_{m_0}(\widetilde{A}+F)$ for every integer $p \geqslant  0$. Then $pL_1 \leqslant L_p$ for all $p$ by construction. 
We denote by $\cU_p$ the direct image sheaf $\varphi_*\cO_\Gamma(pL_1)$.
Then there are natural morphisms of coherent sheaves on $Y$
\[{\Sym}^p \cV_1 \to \cU_p \to \cV_p.\]
Since $L_1$ is sufficiently ample on general fiber of $\varphi$, and since $E''$ does not dominate $Y$, we see that the morphisms above are generically surjective.


We recall the assumption that every $\varphi$-exceptional divisor is $\pi$-exceptional. 
Let $C\subseteq \Gamma$ be the strict transform of a very general complete intersection curve in $X$. 
Then $\varphi^*\mathcal{V}_{p}$ and $\varphi^*\cU_p$ are all locally free around $C$ and $c_1(\varphi^*\mathcal{V}_{p}|_C)=0$.
Let $j\colon C\to Y$ be the natural morphism. 
Then the weakly positively curved  sheaves $j^*\mathcal{V}_{p}$ are numerically flat vector bundles by \cite[Proposition 2.11]{CaoHoering2019}.  
We may assume that they are flat bundles with connections satisfying the conditions of Theorem \ref{thm:numflat=flat}.  
Therefore the generically surjective morphisms
\[ j^*{\Sym}^p\mathcal{V}_{1} \to j^*\cV_{p} \]
must be surjective by Corollary \ref{cor:morphism-flat}.  
As a consequence, the morphisms $$j^*{\Sym}^p \cV_1 \to j^*\cU_p$$ are  surjective morphisms between numerically flat vector bundles as well, and $j^*\cU_p \cong j^*\cV_p$.
Then by Corollary \ref{cor:morphism-flat} and Lemma \ref{lemma:loc-trivial},  the variety
\[U_C=  \mathrm{Proj}_{\cO_{C}} \bigoplus_{p \geqslant 0} j^*\cU_p\]
is a locally trivial  family over $C$.

Let $\Gamma_C=\Gamma \times_Y C$. We consider the following commutative diagram.
\begin{equation*} 
\begin{tikzcd}[column sep=large, row sep=large]
  \Gamma_C \ar[d,"{\varphi}_C"]      \ar[r,"i"]  & \Gamma \ar[d,"\varphi"]   \\
 C \ar[r,"j"] &   Y
\end{tikzcd}
\end{equation*}
There are natural injective morphisms
\[ j^*\cV_p \to (\varphi_C)_*(i^*\cO_\Gamma(L_p))  \mbox{ and }  j^*\cU_p \to (\varphi_C)_*(i^*\cO_\Gamma(pL_1)),\]
which are all generically surjective.
Indeed, $j^*\cU_p$ can be viewed as a linear subsystem of  the $\varphi$-relative complete linear system $(\varphi_C)_*(i^*\cO_\Gamma(pL_1))$.
We obtain hence a  rational map
$$g_C\colon \Gamma_C \dashrightarrow U_C $$ which is induced by the  $\varphi$-relative linear subsystem $j^*\cU_1 \subseteq (\varphi_C)_*(i^*\cO_\Gamma(L_1))$.

Since $L_1$ is the sum of a $\varphi$-relatively sufficiently ample divisor, and an effective $\pi$-exceptional divisor, it follows that the  $\varphi$-relative linear system $\cU_1$ separates points of $\Gamma$ lying  outside the $\pi$-exceptional locus.
Thus, if $D_C$ is a prime divisor in $\Gamma_C$ which is contracted by $g_C$, then $i(D_C)$ is contained in the $\pi$-exceptional locus. 
We also remark that general fibers of $\varphi_C$ and general fibers of $U_C \to C$ are the same. 
Hence every fibers of the locally trivial family $U_C \to C$ is irreducible.  
Thus, for every point $c\in C$, the fiber $\varphi_C^{-1}\{c\}$ has at most one irreducible component which is not $g_C$-exceptional.   
In conclusion, such a fiber has at most  one irreducible component which is not contained in the $\pi$-exceptional locus.    

For a  prime divisor $B \subset Y$, if $B \cap Y_0 \neq \emptyset$, we denote by $B_1,...,B_t$  the irreducible components of $\varphi^*B.$ 
Assume that $B_1$ is not contained in the $\pi$-exceptional locus.  
Then $B_1$ dominates $B$. 
Since $C$ is very general, we may assume that $j(C)$ meets a general point of $B$.   
The previous paragraph then implies that, for $i\neq 1$, either $B_i\cap \Gamma_C$ is contained in the $\pi$-exceptional locus, or $B_i$ is $\varphi$-exceptional. 
In the later case, $B_i$ is also $\pi$-exceptional. 
If $B$ is  contained in $Y\setminus Y_0$, then $\varphi^{-1}(B)$ is contained in the $\pi$-exceptional locus by assumption. 
This completes the proof. 
\end{proof}

\section{Construction of  semistable reductions} \label{section:construction}
Throughout this section, let $\cF$ be an algebraically integrable foliation on a projective manifold $X$ with  $-K_\cF$ nef.  
Assume that $\cF$ has a compact leaf.
Let $f\colon U\to V$ be the family of leaves, with $U$ and $V$ normal.
\begin{equation*} 
\begin{tikzcd}[column sep=large, row sep=large]
  U  \ar[d,"f"]\ar[r,"e"] & X\\
  V &
\end{tikzcd}
\end{equation*} 
Then $f$ is smooth over some open dense subset of $V$. 
The goal of this section is to construct an appropriate semistable reduction, which will serve the proofs  of the main theorems.
We first observe the following statement.

\begin{lemma}\label{lemma:irreducible:fiber}
With the notation above, for any prime divisor $B\subseteq V$, then  $e(f^{-1}(B))$ has at most one irreducible component  of codimension one. 
\end{lemma}

\begin{proof}
By \cite[Theorem 0.3]{AbramovichKaru2000}, weak semistable reductions for the fibration $f$ exist. 
In particular, there is a generically finite projective surjective morphism $\iota\colon Y \to V$ with $Y$ smooth,  and there is  a desingularization $\Gamma$ of the main component of $U\times_V Y$, 
\begin{equation*} 
\begin{tikzcd}[column sep=large, row sep=large]
  \Gamma \ar[d,"{\varphi}"]  \arrow[bend left]{rr}{\rho}   \ar[r,"{}"]  & U  \ar[d,"f"]\ar[r,"e"] & X\\
 Y \ar[r,"\iota"] & V &
\end{tikzcd}
\end{equation*} 
such that for every prime divisor divisor $B_Y \subseteq Y$, any non reduced component of $\varphi^*B_Y$ is $\rho$-exceptional.

We may assume that morphism $\iota$ is  Galois of group $G$ over the generic point of $V$, and there is a natural action of $G$ on $Y$.  
We  assume further that   $\Gamma$ is a $G$-equivariant desingularization, and  the morphisms $\varphi$ is $G$-equivariant. 
We also assume that the $\rho$-exceptional locus is pure of codimension one, and we denote it as a reduced divisor $E$. 
Particularly, $E$ is $G$-invariant.  
Since $f$ is smooth over the generic point of $V$, we may assume that  $\Gamma \to U\times_V Y$ is an isomorphism over the smooth locus. 
In particular, $E$ does not dominate $Y$.

Let $\cF'$ be the foliation on $\Gamma$ induced by $\varphi$. 
On the one hand, by Lemma \ref{lemma:basechange-canonical}, there is some $\pi$-exceptional divisor $E_1$ such that 
\[K_{\cF'} = \rho^* K_{\cF} + E_1.\] 
On the other hand, the semistable assumption implies that, there is an open subset $W\subseteq \Gamma$ on which $\varphi$ is smooth, such that every component of $\Gamma \setminus W$  is either a $\pi$-exceptional divisor, or of codimension at least two in $\Gamma$. 
Thus there is some $\pi$-exceptional divisor $E_2$ such that 
\[K_{\cF'} = K_{\Gamma/Y} + E_2.\]
In conclusion, there is a  $\pi$-exceptional divisor $E''$ such that 
\begin{equation*}
-K_{\Gamma/Y}+ E''=\rho^*(-K_\cF), 
\end{equation*}
which is  nef by assumption. 
Let $\pi\colon \Gamma \to X'$ be the Stein factorization of $\rho$.  
We note that $X'/G = X$, which is smooth. 
Therefore,  we can now apply the techniques in Section \ref{section:semistable} to the diagram 
\begin{equation*} 
\begin{tikzcd}[column sep=large, row sep=large]
  \Gamma \ar[d,"{\varphi}"]     \ar[r,"\pi"]  & X'  \\
 Y  & 
\end{tikzcd}
\end{equation*}

Let $B\subseteq V$ be a prime divisor and $B'$ be a component of $\iota^{-1}(B)$ which dominates $B$. 
By Proposition \ref{prop:irred-fiber-semistable},   there is at most one component of $\varphi^{-1}(B')$ which is not in the $\pi$-exceptional locus.  
Since $\Gamma$ is a desingualrization  of the main component of the fiber product $Y\times_V U$, this implies that  the preimage $f^{-1}(B)$ has at most one  divisorial component which is not contained in the  $e$-exceptional locus.
\end{proof}

Next we will prove the following construction.

\begin{lemma}
\label{lemma:good-semistable-reduction}
There is a generically finite, projective surjective morphism $\iota\colon Y \to V$ with $Y$ smooth, such that if $\Gamma$ is  a desingularization  of the main component of $U\times_V Y$,  
\begin{equation*} 
\begin{tikzcd}[column sep=large, row sep=large]
  \Gamma \ar[d,"{\varphi}"]  \arrow[bend left]{rr}{\rho}   \ar[r,"{}"]  & U  \ar[d,"f"]\ar[r,"e"] & X\\
 Y \ar[r,"\iota"] & V &
\end{tikzcd}
\end{equation*} 
then  for every prime divisor $B_Y\subseteq Y$, the preimage $\varphi^*B_Y$ has at least one reduced component which dominates $B_Y$,  and any non reduced component of $\varphi^*B_Y$ is $\rho$-exceptional. 
Moreover, for any prime divisor $B$ in $V$,  contained in the discriminant of $\iota$,
\begin{enumerate}
\item either $e(f^{-1}(B))$ has codimension at least two in $X$,
\item or $f$ is smooth over the generic point of $B$,
\item or $e(f^{-1}(B))$ has a unique  irreducible component $D$ of codimension one, and $\rho$ is \'etale  over the generic point of $D$.
\end{enumerate} 
Furthermore, $\iota$ is Galois of group $G$ over the generic point of $V$, and there is a natural action of $G$ on $Y$.  
\end{lemma}

We recall that the discriminant of $\iota$ is the locus in $V$ over which $\iota$ is not smooth.  
With this construction, we have the following property. 

\begin{cor}\label{cor:good-semistable-reduction} 
If $D$ is a prime divisor in $X$ contained in the discriminant of $\rho$, then  there is a prime divisor $B$ in $V$ such that  $f$ is smooth over the generic point of $B$, and that $D$ is the unique component of $e(f^{-1}(B))$ which has codimension one.
\end{cor}

\begin{proof}
Let $D_U$ be the strict transform of $D$ in $U$. 
Then the natural morphism $\Gamma \to U$ is branched at the generic point of $D_U$. 
Since $\Gamma$ is a desingularization of the fiber product $Y \times_V U$, we see that $D_U$ is vertical over $V$. 
Then $B=f(D_U)$ is a prime divisor in $V$ since $f$ is equidimensional.  
Thus  we are in the case of item (2) of   Lemma \ref{lemma:good-semistable-reduction}. 
This completes the proof of the corollary.  
\end{proof}

\begin{proof}[{ Proof of Lemma \ref{lemma:good-semistable-reduction} }]
We note that $f$ is smooth over the generic point of $V$. 
Let $V' \to V$ be a desingularization such that  there is a simple normal crossing divisor $\Delta$ in $V'$ such that  the natural morphism $f'\colon U' \to V'$ is smooth over $V' \setminus \mathrm{Supp}\,\Delta$. Here $U'$ is the normalization of the main component of $U\times_V V'$. 

Let $B_1,...,B_k$ be all the components of $\Delta$ such that $f'^*B_i$ has a component $D_i$  which is not exceptional over $X$.  
Lemma \ref{lemma:irreducible:fiber} shows that such a component $D_i$ is unique in $f'^*B_i$. 
Let $B_{k+1},...,B_{l}$ be the other components of $\Delta$. Let $D_{j}$ be an irreducible component of $f'^*B_j$   for $j=k+1,...,l$. 
We denote by $m_i$ the multiplicity of $D_i$ in $f'^*B_i$ for $i=1,...,l$.

By Kawamata's covering trick (see \cite[Theorem 17]{Kawamata1981} or \cite[Theorem 1-1-1]{KMM87}),  there is a finite Galois cover $\sigma\colon Y \to V'$ with $Y$ smooth such that 
$\sigma^*B_i$ is pure of multiplicity $m_i$ for $i=1,...,l$.
Let $\Gamma$ be a desingularization of   $Y\times_{V'} U'$. 
Then by construction,  the natural morphism $\Gamma \to U'$ is \'etale over the generic point of $D_i$ for all $i$ (see for example \cite[Lemma A.3]{DasOu2022}). 
Moreover,  for every prime divisor $B_Y\subseteq Y$, the preimage $\varphi^*B_Y$ has at least one reduced component which dominates $B_Y$. 
We also see that any non reduced component of $\varphi^*B_Y$ is $\rho$-exceptional.

Let $\iota\colon Y \to V$ be the morphism. 
To complete the proof, let  $B$ be a prime divisor in $V$, contained in the discriminant of $\iota$. Assume  that $e(f^{-1}(B))$ has an irreducible component $D$  of codimension one. Then $D$ is unique by Lemma  \ref{lemma:irreducible:fiber}. 
Assume furthermore that $f$ is not smooth over the generic point of $B$.  
Then we need to prove that $B$ satisfies item (3) of the Lemma.
From the second paragraph,  the strict transform of $B$ in $V'$ is one of the element in $\{B_1,...,B_k\}$, say $B_1$. 
The previous paragraph implies that $\Gamma \to U'$ is \'etale over the generic point of $D_1$. 
Since $D_1$ is the strict transform of $D$ in $U'$, we conclude that $\rho\colon \Gamma \to X$ is \'etale over the generic point of $D$.  
\end{proof}

\section{Proofs of the main theorems}\label{section:proof}

In this section, we will finish the proofs of the theorems in the introduction.  We first recall the demonstration of Theorem \ref{thm:rc-reduction}. 

\begin{proof}[{Proof of Theorem \ref{thm:rc-reduction}}]
Thanks to \cite[Lemma 6.2]{Druel2017}, the algebraic part $\cF_{alg}$ of $\cF$ has a compact leaf.  
This proves item (1).   
We can now apply \cite[Proposition 6.1]{Druel2017} to show that 
$K_{\cF} \equiv  K_{\cF_{alg}}$.

As in  \cite[Claim 4.3]{Druel2019}, we can then deduce that $K_{\cF} \equiv K_{\cF_{rc}}$.  
We note that the foliation $\cF_{rc}$ here  corresponds to the foliation $\cH$ in \cite[Claim 4.3]{Druel2019}, which was defined at the end of \cite[page 316]{Druel2019}.    
Furthermore, in  \cite[Claim 4.3]{Druel2019}, the foliation $\cF$ is assumed to have semipositive anticanoncial class.  
Nevertheless, its proof only uses the fact that $-K_\cF$ is nef.  
\end{proof}

We will give the proofs of  Corollary \ref{cor:structure}, Theorem \ref{thm:main-theorem} and Corollary \ref{cor:regularity}  by  assuming Theorem \ref{thm:decomposition}.

\begin{proof}[{Proof of Corollary \ref{cor:structure}}]
By Theorem \ref{thm:decomposition}, $\cF$ is a direct summand of $T_X$.  
Then \cite[Proposition 2.5]{HwangViehweg2010} implies that, there is a morphism  $\mu$ from $X$ to  the Chow scheme of $X$, such that $\mu$ sends a leaf $C$ of $\cF$ to the  cycle class $[|G_C|C]$, where $|G_C|$ is the cardinality of the holonomy group of $C$.   
Since $X$ is irreducible, the image of $\mu$ is irreducible as well. 
This shows that every leaf of $\cF$  has the same dimension.  

If we assume further that general leaves of $\cF$ are rationally connected, then $f$ is a smooth morphism by \cite[Corollary 2.11]{Hoering07}.  
Then this fibration  is  a locally trivial family over $Y$  (see Corollary \ref{cor:smooth-morphism}).
\end{proof}

\begin{example}\label{example:multiple-fibers}
If the leaves of $\cF$ are not rationally connected, then the fibration $f\colon X \to Y$ in Corollary \ref{cor:structure} might not be smooth. Let $G=\mathbb{Z}/2\mathbb{Z}$. 
Then there is a $G$-action on $\mathbb{P}^1$ which has exactly two fixed points. 
Let $E$ be an elliptic curve equipped with a free $G$-action. 
We set $X=(E\times\mathbb{P}^1)/{G}$, where the quotient is with respect to the diagonal action. 
Then $X$ is smooth. There is a natural fibration $f\colon X\to Y=\mathbb{P}^1/G$. 
Then $f$ is not smooth and the foliation in curves  induced by $f$ is nef. 
\end{example}

\begin{proof}[{ Proof of Theorem \ref{thm:main-theorem}}]
We note that $\cF_{rc}$ has a compact leaf. 
By Theorem \ref{thm:rc-reduction}, we get  $-K_{\cF_{rc}} \equiv -K_{\cF}$ is nef. 
Hence by Corollary \ref{cor:structure},   $\cF_{rc}$ is induced by a smooth fibration $f\colon X\to Y$.  
Moreover, this fibration  is a locally trivial family over $Y$. 
It follows that there is a foliation $\cG$ on $Y$ such that $\cF=f^{-1}\cG$ (see for example \cite[Lemma 6.7]{AraujoDruel2013}).
From $-K_{\cF_{rc}} \equiv -K_{\cF}$, we deduce that $K_\cG \equiv 0$. This completes the proof of the theorem.
\end{proof}

\begin{proof}[{Proof of Corollary \ref{cor:regularity}}]
By Theorem \ref{thm:main-theorem},  there is a locally trivial fibration $f\colon X\to Y$  
and  a foliation $\cG$ on $Y$ with $K_{\cG} \equiv 0$ such that $\cF = f^{-1}\cG$. 
Since $\cF$ has a compact leaf, so has  $\cG$. 
Hence  by \cite[Theorem 5.6]{LorayPereiraTouzet2018}, $\cG$ is regular, and $T_Y=\cG \oplus \cH$ for some regular foliation $\cH$ on $Y$.  
In particular, $\cF$ is regular as well. 

Let $\cF_f$ be the foliation induced by $f$. We observe that $-K_{\cF_f} \equiv -K_\cF$  is  nef. 
By Theorem \ref{thm:decomposition}, there is a regular foliation $\cK$ on $X$ such that $T_X=\cK \oplus \cF_f$. 
Then the natural map $\cK\to f^*T_Y$ is an isomorphism. 
We set $\cE=\cK\cap f^{-1}\cH$. 
Then $\cE$ is a regular foliation and we verify that $T_X=\cF \oplus \cE$.   
\end{proof}

The remainder of this section is then devoted to the proof of Theorem \ref{thm:decomposition}.  We will divide it into several steps.

\subsection{Setup}\label{section:setting}
Until the end of the section, we will consider  the following situation.  Let $\cF$ be an algebraically integrable foliation with nef anticanoncial class $-K_{\cF}$, on  a projective manifold $X$.  
Assume that $\cF$ has a compact leaf.   
Let $f\colon U\to V$ be the family of leaves  with $U$ and $V$ normal.  
We take  a generically finite  morphism  $\iota \colon Y\to V$, and a  desingularization $\Gamma$ of the main component of $U\times_V Y$,  as in Lemma \ref{lemma:good-semistable-reduction}.  
\begin{equation*} 
\begin{tikzcd}[column sep=large, row sep=large]
  \Gamma \ar[d,"{\varphi}"]  \arrow[bend left]{rr}{\rho}   \ar[r,"{}"]  & U  \ar[d,"f"]\ar[r,"e"] & X\\
 Y \ar[r,"\iota"] & V &
\end{tikzcd}
\end{equation*}
Then for every prime divisor $B_Y\subseteq Y$, the preimage $\varphi^*B_Y$ has at least one reduced component which dominates $B_Y$, and any non reduced component of $\varphi^*B_Y$ is $\rho$-exceptional. 
Moreover, for any prime divisor $B$ in $V$,  contained in the discriminant of $\iota$,
\begin{enumerate} 
\item either $e(f^{-1}(B))$ has codimension at least two in $X$,
\item or $f$ is smooth over the generic point of $B$,
\item or $e(f^{-1}(B))$ has a unique  irreducible component $D$ of codimension one, and $\rho$ is \'etale  over the generic point of $D$.
\end{enumerate}

We note that, by construction, every $\varphi$-exceptional divisor is  $\rho$-exceptional  as well.  
The morphism $\iota$ is  Galois of group $G$ over the generic point of $V$, and there is a natural action of $G$ on $Y$.  
We  assume further that   $\Gamma$ is a $G$-equivariant desingularization, and  the morphisms $\varphi$ is $G$-equivariant. 
We also assume that the $\rho$-exceptional locus is pure of codimension one, and we denote it as a reduced divisor $E$.
Since $f$ is smooth over the generic point of $V$, we may assume that  $\Gamma \to U\times_V Y$ is an isomorphism over the smooth locus. 
In particular, $E$ does not dominate $Y$.
Let $\pi\colon \Gamma \to X'$ the  Stein factorization of $\rho$.  
We note that $X'/G = X$, and we have the following diagram.

\begin{equation*} 
\begin{tikzcd}[column sep=large, row sep=large]
  \Gamma \ar[d,"{\varphi}"]  \arrow[bend left]{rr}{\rho}   \ar[r,"\pi"]  & X' \ar[r,"q"] & X\\
 Y  & &
\end{tikzcd}
\end{equation*}

 We denote by $\cF'$ the foliation on $\Gamma$ induced by $\varphi$. On the one hand, by Lemma \ref{lemma:basechange-canonical}, there is some $\pi$-exceptional divisor $E_1$ such that 
\[K_{\cF'} = \rho^* K_{\cF} + E_1.\] 
On the other hand, by the semistable assumption, there is some $\pi$-exceptional divisor $E_2$ such that 
\[K_{\cF'} = K_{\Gamma/Y} + E_2.\]
In conclusion, there is a $G$-invariant $\pi$-exceptional divisor $E''$ such that 
\begin{equation*}
-K_{\Gamma/Y}+ E''=\rho^*(-K_\cF).
\end{equation*}
Since $-K_{\cF}$ is assumed to be nef, we can   apply the techniques in Section \ref{section:semistable} to the diagram 
\begin{equation*} 
\begin{tikzcd}[column sep=large, row sep=large]
  \Gamma \ar[d,"{\varphi}"]     \ar[r,"\pi"]  & X'  \\
 Y  & 
\end{tikzcd}
\end{equation*}

Let  $V_0 \subseteq V$ be a  Zariski open  subset such that the following properties hold,
\begin{enumerate} 
\item[$\bullet$] $\iota$ is finite over $V_0$;
\item[$\bullet$] for any prime divisor $B\subseteq V $, the preimage $f^{-1}(B)$ is   contained in the $e$-exceptional locus  if and only if $B\subseteq V \setminus V_0$.  
\end{enumerate}
Let $Y_0=\iota^{-1}(V_0)$.  We pick a sufficiently ample   divisor $A$ on $\Gamma$.   
We will assume that $A$ is $G$-invariant. 
More precisely, we assume that  \[A=\rho^*N + E_A\] where $N$ is a sufficiently ample divisor on $X$, and $E_A$ is $\rho$-exceptional and $G$-invariant.
We define the following $\mathbb{Q}$-divisor, with $m_0$ and $r_{m_0}$ as in Lemma \ref{lemma:A-tilde},  
\[\widetilde{A} = A+m_0E - \frac{1}{r_{m_0}}\varphi^*c_1(\varphi_*\cO_\Gamma (A+m_0E)).\]
We remark that, $\widetilde{A}$ is $G$-invariant as well.  
By Lemma \ref{lemma:normalized-c_1}, there is a an effective divisor such that $\widetilde{A}+F$ is pseudoeffective. 
Replacing $F$ by the sum of its $G$-orbit, we may assume that $F$ is 
$G$-invariant. 

As in Corollary \ref{cor:Vp}, we choose a sequence  $\{c_p\}_{p\geqslant 0}$ of nonnegative integers with $c_0=0$ such that  the  direct image sheaves 
\[\mathcal{V}_{p} = \varphi_*\cO_\Gamma (c_{p}|E''|+pr_{m_0}(\widetilde{A}+F))\]
are weakly positively curved. 
The first Chern  class   $c_1((\varphi^*\cV_p)^{**})$ is supported in the $\rho$-exceptional locus for every $p$. 
Furthermore  $c_i+c_j \leqslant c_{i+j}$ for any $i,j \geqslant 0$.
By construction, every $\cV_p$ is $G$-linearized.

\subsection{Flatness about $\cV_p$} \label{section:flatness}
 
If  $X'$ is smooth, then the reflexive hull $(\pi_*(\varphi^*\cV_p))^{**}$ will be a numerically flat vector bundle, as it is weakly positively curved with zero first Chern class (see  \cite[Proposition 2.6]{CCM19}).
In general  $X'$  is not smooth, and the smoothness condition lies on $X$ instead. 
Hence, we aim to descend $(\varphi^*\cV_p)^{**}$ to $X$.
 
\begin{equation*} 
\begin{tikzcd}[column sep=large, row sep=large]
  \Gamma \ar[d,"{\varphi}"]  \arrow[bend left]{rr}{\rho}   \ar[r,"\pi"]  & X' \ar[r,"q"] & X\\
 Y  & &
\end{tikzcd}
\end{equation*}

\begin{lemma}
\label{lemma:descend-V_p}
 Let $X^\circ \subseteq X$ be the largest Zariski open subset such that $\rho$ is a finite morphism over $X^\circ$. 
 Let $\Gamma^\circ = \rho^{-1}(X^\circ)$.
Then there is a reflexive sheaf $\cE_p$ on $X$ such that $(\rho^*\cE_p)^{**}|_{\Gamma^\circ} \cong (\varphi^*\cV_p)^{**}|_{\Gamma^\circ}$. 
\end{lemma} 
 
\begin{proof}
We note that $\Gamma^\circ/G \cong X^\circ$ and that $(\varphi^*\cV_p)^{**}$ is $G$-linearized.    
To prove the existence of $\cE_p$, we need to verify that $(\varphi^*\cV_p)^{**}$ satisfies Kempf's descending condition (see \cite[Th\'eor\`eme 2.3]{DrezetNarasimhan1989}).  
More precisely, let $\Delta$ be a prime divisor in $X$, contained in the discriminant of  $\rho$.  
It is enough to show that, for a general point $x\in \Delta$,  a point $\gamma \in \Gamma$ lying over $x$, and a stabilizer $\theta \in G_\gamma$, the action of $\theta$ on  the fiber $(\varphi^*\cV_p)^{**}_\gamma$ is trivial.

We note that,  the divisor $\Lambda=f(e_*^{-1}\Delta)$ is contained in the discriminant of $\iota\colon Y \to V$. By construction (see Lemma \ref{lemma:good-semistable-reduction} and Corollary \ref{cor:good-semistable-reduction}), the fibration $f$ is smooth over the generic point of $\Lambda$. Particularly, there is an open neighborhood $S$ of the generic point of $\Lambda$ such that $\Gamma$ is equal to the fiber product $U\times_V Y$ over $S$. Let $T = \iota^{-1}(S)$. 
We may assume further that $\cV_p$ is locally free on $T$.

\begin{equation*} 
\begin{tikzcd}[column sep=large, row sep=large]
  \Gamma \ar[d,"{\varphi}"]  \arrow[bend left]{rr}{\rho}   \ar[r,"{}"]  & U  \ar[d,"f"]\ar[r,"e"] & X\\
 Y \ar[r,"\iota"] & V &
\end{tikzcd}
\end{equation*}

We recall that,  
$$\cV_p=\varphi_*\cO_{\Gamma}(c_p|E''|+pr_{m_0}(\widetilde{A}+F)),$$  
$$\widetilde{A}=A+m_0E-\frac{1}{r_{m_0}}\varphi^*c_1(\varphi_*\cO_\Gamma (A+m_0E)),$$ 
and $A=\rho^*N + E_A$. 
Let $$\widetilde{N}= N  -\frac{1}{r_{m_0}} f^* c_1(f_*\cO_U (e^*N)) $$
be a $\mathbb{Q}$-divisor in $U$.
Since $T\to S$ is flat, and since the $\rho$-exceptional locus  does not meet $\varphi^{-1}(T)$,  we obtain that 
\[\cV_p|_T \cong \iota^*(f_*\cO_U (pr_{m_0}\widetilde{N}))|_T.\] 
Moreover, such an isomorphism is $G$-equivariant. 
Since $x$ is  general in $\Delta$, the point $\gamma$ lies in $\varphi^{-1}(T)$. 
Hence the action of $\theta$ on $(\varphi^*\cV_p)^{**}_\gamma$ is trivial. This completes the proof  of the lemma.
\end{proof}

\begin{lemma}\label{lemma:E_p-Num-flat}
With the notation above, each $\cE_p$ is a numerically flat vector bundle. In particular, we may  assume that $\cE_p$ is equipped with the flat connection satisfying the conditions of Theorem \ref{thm:numflat=flat}.
\end{lemma}

\begin{proof}
We fix some $p\geqslant 0$. For simplicity, we write $\cE=\cE_p$ and $\cV=\cV_p$. 
We first remark that $c_1(\cE)=0$ since $c_1((\varphi^*\cV)^{**})$ is supported in the $\rho$-exceptional locus.

By shrinking $Y_0$ if necessary, we may assume that $\cV$ is locally free on $Y_0$. 
We recall that $\Gamma^\circ\subseteq \Gamma$ is the largest open subset on which $\rho$ is finite.
By assumption, $Y_0\setminus \varphi(\Gamma^\circ)$ has codimension at least $2$, and $\varphi$ is equidimensional over $Y_0$.
Set $\Gamma^\circ_0 = \Gamma^\circ \cap \varphi^{-1}(Y_0)$. 
Then the complement of $\pi(\Gamma^\circ_0)$ has codimension at least two in $X'$.
Let $X_0'$ be the largest $G$-invariant open subset of $\pi(\Gamma^\circ_0)$, on which $q^*\cE$ is locally free. 
Then its complement has codimension at least two in $X'$ as well. 
Moreover, the rational map $\psi \colon X' \dashrightarrow Y$ induces a morphism $\psi\colon X'_0 \to Y_0$, such that we have an isomorphism  $(q^*\cE)|_{X_0'} \cong (\psi^*\cV)|_{X_0'}$. Let $X_0 = q(X'_0).$  We have the following commutative diagram.
\begin{equation*} 
\begin{tikzcd}[column sep=large, row sep=large]
    \mathbb{P}(\cV|_{Y_0}) \ar[d] &  \mathbb{P}(q^*\cE|_{X'_0}) \ar[r ] \ar[l]  \ar[d  ] \ar[r,"\mathrm{finite}"] &  \mathbb{P}(\cE|_{X_0})  \ar[d] \\
  Y_0      & X'_0\ar[l,"\psi"] \ar[r,"q"'] & X_0 
\end{tikzcd}
\end{equation*}  
Since $\cV$ is weakly positively curved, by Lemma \ref{lemma:weakly-positive}, for any ample divisor $A_Y$ in $Y$, any positive integer $a$, there is some integer $m>0$ such that  the natural evaluation map
\[H^0(Y,{(\Sym}^{am} \cV)^{**} \otimes \cO_X(mA_Y)) \to  ({\Sym}^{am} \cV)^{**}_y \otimes \cO_X(mA_Y)_y \]
 is surjective  for general $y\in Y$.
 
Therefore, for every ample divisor $A_X$ in $X$, any positive integer $a$, there is some integer $m>0$ such that   
\[H^0(X'_0, {\Sym}^{am} (q^*\cE) \otimes \cO_X(mq^*A_X)) \to  {\Sym}^{am} (q^*\cE)_x \otimes \cO_X(mq^*A_X)_x \]
 is surjective  for general $x\in X'_0$.
As a consequence, the diminished base locus of $\cO_{\mathbb{P}(q^*\cE|_{X'_0})}(1)$ is not surjective onto $X'_0$.  
Hence the diminished base locus of $\cO_{\mathbb{P}(\cE|_{X_0})}(1)$ is not surjective onto $X_0$ neither, for $q$ is a finite morphism.  
Thus $\cE$ is numerically flat by Lemma \ref{lemma:num-flat-criterion-2}.  
\end{proof}

\begin{lemma}\label{lemma:Vp-Wp}
There is a  flat vector bundle $\cW_p$ on $Y$ such that $\rho^*\cE_p = \varphi^*\cW_p$. Moreover, there is a natural isomorphism between $\cW_p|_{Y_0}$  and $\cV_p^{**}|_{Y_0}$.
\end{lemma}

\begin{proof}
By assumption  $\rho^*\cE_p$ is equipped with the   flat connection $\nabla_{\rho^*\cE_p}$ of Theorem \ref{thm:numflat=flat}.
Since $\Gamma^\circ$ contains general fibers of $\varphi$,  by the definition of $\cE_p$ (see Lemma \ref{lemma:descend-V_p}),  the restriction of $\rho^*\cE_p$ to a general fiber of $\varphi$  is isomorphic  to a  trivial bundle  as holomorphic vector bundles.
From the uniqueness of Theorem \ref{thm:numflat=flat}, we obtain that $\nabla_{\rho^*\cE_p}$ induces the trivial connection on the restriction of $\rho^*\cE_p$ to a general fiber of $\varphi$.
Therefore, the representation of the fundamental  group  $\pi_1(\Gamma)$ corresponding to  $(\rho^*\cE_p,\nabla_{\rho^*\cE_p})$ induces trivial representation on the fundamental group of a general fiber of $\varphi$.
Then Proposition  \ref{prop:flat-descend} implies that there is a  flat vector bundle $\cW_p$ on $Y$ such that $\rho^*\cE_p = \varphi^*\cW_p$.

Finally, we remark that   $Y_0 \setminus \varphi(\Gamma^\circ)$   has codimension at least $2$ in $Y_0$. Hence  Lemma \ref{lemma:iso-descend} implies that there is a natural isomorphism between $\cW_p|_{Y_0}$  and $\cV_p^{**}|_{Y_0}$.
\end{proof}

\subsection{Decomposition of the tangent bundle}
Let $\nabla_{\cW_p}$ be the  flat connection on $\cW_p$, induced by the flat connection on $\cE_p$.  
Then it satisfies the condition of Theorem \ref{thm:numflat=flat}, and is $G$-equivariant.

\begin{equation*} 
\begin{tikzcd}[column sep=large, row sep=large]
  \Gamma \ar[d,"{\varphi}"]  \arrow[bend left]{rr}{\rho}   \ar[r,"\pi"]  & X' \ar[r,"q"] & X\\
 Y  & &
\end{tikzcd}
\end{equation*}

\begin{lemma}\label{lemma:morphism-Wp}
There is a surjective morphism of commutative graded $\cO_Y$-algebras 
\[ \bigoplus_{p\geqslant 0} {\Sym}^p\cW_1 \to \bigoplus_{p\geqslant 0} \cW_p.\]
Moreover, each graded piece is a morphism of flat bundles.  
\end{lemma} 
 
\begin{proof}
For any integers $p,i,j\geqslant 0$, 
there are natural morphisms  ${\Sym}^p \cV_1 \to \cV_p$ and $\cV_i \otimes \cV_j \to \cV_{i+j}$. 
They are generically surjective since $\widetilde{A}$ is sufficiently ample on a general fiber of $\varphi$.  
Furthermore, these morphisms are all $G$-equivariant.
We can then obtain  generically surjective morphisms
\[ {\Sym}^p\cE_1 \to \cE_p  \mbox{ and } \cE_i \otimes \cE_j \to \cE_{i+j}. \]
By Corollary \ref{cor:morphism-flat}, they are surjective morphisms of flat vector bundles.

Moreover, the following  natural diagram is commutative,
\begin{equation*} 
\begin{tikzcd}[column sep=large, row sep=large]
    {\Sym}^i \cV_1 \otimes {\Sym}^j \cV_1   \ar[d," "]      \ar[r," "]  & \cV_i\otimes \cV_j  \ar[d," "]   \\
  {\Sym}^{i+j} \cV_1 \ar[r," "] &   \cV_{i+j}
\end{tikzcd}
\end{equation*} 
We then deduce that 
 \[ \bigoplus_{p\geqslant 0} {\Sym}^p\cE_1 \to \bigoplus_{p\geqslant 0} \cE_p \]
is  a surjective morphism of commutative graded $\cO_X$-algebras. 

By construction, we have $\rho^*\cE_p = \varphi^*\cW_p$ for any $p$.
Hence we obtain an induced surjective morphism of commutative graded $\cO_Y$-algebras 
\[ \bigoplus_{p\geqslant 0} {\Sym}^p\cW_1 \to \bigoplus_{p\geqslant 0} \cW_p.\]
This completes the proof of the lemma.
\end{proof}

We are now ready to finish the proof of  Theorem \ref{thm:decomposition}.

\begin{proof}[{ Proof of Theorem \ref{thm:decomposition} }]
By Lemma \ref{lemma:morphism-Wp} and by Lemma \ref{lemma:loc-trivial},  the variety
\[Z = \mathrm{Proj}_{\cO_Y} \bigoplus_{p\geqslant 0} \cW_p\] is a locally trivial family over $Y$. 
Furthermore  the connection $\nabla_{\cW_1}$ induces a foliation  $\cG_Z$ on $Z$,  
which is  transverse to the fibers of $Z\to Y$.

By removing some closed subset of codimension at least two, we assume that $\cV_1$ is locally free on $Y_0$.  
By Lemma \ref{lemma:Vp-Wp}, there is a natural isomorphism  between $\cW_p|_{Y_0}$  and $\cV_p^{**}|_{Y_0}$.  
Since ${\Sym}^p \cW_1 \to \cW_p$ is surjective, we conclude that the natural morphism $${\Sym}^p \cV_1|_{Y_0} \to \cV_p|_{Y_0}$$ is surjective, 
and  $\cV_p|_{Y_0}  \cong \cW_p|_{Y_0}$ is locally free for any $p\geqslant 0$.

We recall that $\cV_p=\varphi_*\cO_{\Gamma}(L_p)$, where 
$$L_p=c_p|E''|+pr_{m_0}(\widetilde{A}+F).$$
In particular, the morphism ${\Sym}^p\cV_1 \to \cV_p$ factors through $\varphi_*\cO_\Gamma(pL_1)$.
Therefore, we obtain isomorphisms  $$\varphi_*\cO_\Gamma(pL_1)|_{Y_0} \cong \cV_p|_{Y_0} \cong \cW_p|_{Y_0}.$$
Hence, there is a natural  rational map 
\[\psi \colon \Gamma \dashrightarrow  Z, \] 
induced by the $\varphi$-relative linear system $\cV_1$ of  the line bundle $\cO_\Gamma(L_1)$.
Since $L_1$ is the sum of a $\varphi$-relatively sufficiently ample divisor and an effective $\rho$-exceptional divisor,
we see that $\psi$ only contracts $\rho$-exceptional divisors. 
In particular, up to a closed subset of codimension at least two, the restriction $$\psi|_{\Gamma^\circ} \colon \Gamma^\circ \dashrightarrow Z$$ is an isomorphism to its image.  
We recall that $\Gamma^\circ\subseteq \Gamma$ is the largest open subset such that $\rho|_{\Gamma^\circ}$ is finite.

The foliation $\cG_Z$  induces a foliation $\cG_\Gamma$ on $\Gamma$. 
Then $\cG_\Gamma$ is transverse to $\cF'$ over $\Gamma^\circ$.
We recall that $\cF'$ is the foliation induced by $\varphi\colon  \Gamma \to Y$.
Furthermore, $\cG_Z$ is $G$-linearized as it is induced by  the $G$-equivariant connection  $\nabla_{\cW_1}$  (see Remark \ref{remark:Ehresmann-connection}).
Hence $\cG_\Gamma$ descends to a foliation $\cG$ on $X$.

We will show that $T_X = \cF \oplus \cG.$ For the reason of ranks, it is equivalent to show that $\cF$ and $\cG$ are transverse.
It is enough to show this in  codimension one of $X$.  
First we note that, on $X_1 \subseteq X$, the largest open subset over which $\rho$ is \'etale,  they are transverse, for their pullback are transverse in $\rho^{-1}(X_1) \subseteq \Gamma^\circ$. 

Now we consider a prime divisor $D$ in $X$ contained in the discriminant of $\rho$. 
Then by construction (see Lemma \ref{lemma:good-semistable-reduction} and Corollary  \ref{cor:good-semistable-reduction}), there is a prime divisor $B$ in $V$ such that $f$ is smooth over the generic point of $B$, and that $D$ is the unique  component of  $e(f^*B) $ which has codimension one.
\begin{equation*} 
\begin{tikzcd}[column sep=large, row sep=large]
  \Gamma \ar[d,"{\varphi}"]  \arrow[bend left]{rr}{\rho}   \ar[r,"{}"]  & U  \ar[d,"f"]\ar[r,"e"] & X\\
 Y \ar[r,"\iota"] & V &
\end{tikzcd}
\end{equation*}
Let  $V_1$ be an open neighborhood  of the generic point of $B$ such that $f$ is smooth over $V_1$.  
We write $Y_1=\iota^{-1}(V_1)$  and $U_1=f^{-1}(V_1)$. 
There is an open subset $U_2\subseteq U_1$ such that $e|_{U_2}$ is an isomorphism onto its image and that $e(U_2)$ contains the generic point of $D$. 
We remark that  $\varphi^{-1}(Y_1)  \cong U_1\times_{V_1} Y_1$ and  that  $ U_2\times_{V_1} Y_1 \subseteq \Gamma^\circ$. 
Since $\cG_\Gamma$ is transverse to $\cF'$ over $\Gamma^\circ$, by Lemma \ref{lemma:descend-splitting}, we deduce that $\cG|_{e(U_2)}$ is transverse to $\cF|_{e(U_2)}$. 
This completes the proof of the theorem.
\end{proof}

\renewcommand\refname{Reference}
\bibliographystyle{alpha}
\bibliography{foliation_nef}

\end{document}